\documentclass[11pt]{amsart}
\usepackage{amssymb,amsmath,epsfig,mathrsfs, enumerate, xparse, mathtools}
\usepackage[pagewise]{lineno}
\usepackage[nodisplayskipstretch]{setspace}
\usepackage{graphicx}\usepackage[normalem]{ulem}
\usepackage{fancyhdr}
\usepackage[margin=2.5cm]{geometry}
\usepackage{hyperref}
\hypersetup{
 colorlinks   = true,
 urlcolor     = blue,
 linkcolor    = blue,
 citecolor   = red ,
 bookmarksopen=true
}

\theoremstyle{plain}
\newtheorem{thrm}{Theorem}[section]
\newtheorem{lemma}[thrm]{Lemma}
\newtheorem{prop}[thrm]{Proposition}
\newtheorem{cor}[thrm]{Corollary}

\allowdisplaybreaks
\begin{document}
\newcommand{\sn}{\mathbb{S}^{n-1}}
\newcommand{\SL}{\mathcal L^{1,p}( D)}
\newcommand{\Lp}{L^p( Dega)}
\newcommand{\py}{  \partial_y^a}
\newcommand{\La}{\mathscr{L}_a}
\newcommand{\CO}{C^\infty_0( \Omega)}
\newcommand{\Rn}{\mathbb R^n}
\newcommand{\Rm}{\mathbb R^m}
\newcommand{\R}{\mathbb R}
\newcommand{\Om}{\Omega}
\newcommand{\Hn}{\mathbb H^n}
\newcommand{\aB}{\alpha B}
\newcommand{\eps}{\ve}
\newcommand{\BVX}{BV_X(\Omega)}
\newcommand{\p}{\partial}
\newcommand{\IO}{\int_\Omega}
\newcommand{\bG}{\boldsymbol{G}}
\newcommand{\bg}{\mathfrak g}
\newcommand{\bz}{\mathfrak z}
\newcommand{\bv}{\mathfrak v}
\newcommand{\Bux}{\mbox{Box}}
\newcommand{\e}{\ve}
\newcommand{\X}{\mathcal X}
\newcommand{\Y}{\mathcal Y}
\newcommand{\W}{\mathcal W}
\newcommand{\la}{\lambda}
\newcommand{\vf}{\varphi}
\newcommand{\rhh}{|\nabla_H \rho|}
\newcommand{\Ba}{\mathcal{B}_\beta}
\newcommand{\Za}{Z_\beta}
\newcommand{\ra}{\rho_\beta}
\newcommand{\n}{\nabla}
\newcommand{\vt}{\vartheta}
\newcommand{\its}{\int_{\{y=0\}}}

\numberwithin{equation}{section}

\newcommand{\RN} {\mathbb{R}^N}
\newcommand{\Sob}{S^{1,p}(\Omega)}
\newcommand{\Dxk}{\frac{\partial}{\partial x_k}}
\newcommand{\Co}{C^\infty_0(\Omega)}
\newcommand{\Je}{J_\ve}
\newcommand{\beq}{\begin{equation}}
\newcommand{\bea}[1]{\begin{array}{#1} }
\newcommand{\eeq}{ \end{equation}}
\newcommand{\ea}{ \end{array}}
\newcommand{\eh}{\ve h}
\newcommand{\Dxi}{\frac{\partial}{\partial x_{i}}}
\newcommand{\Dyi}{\frac{\partial}{\partial y_{i}}}
\newcommand{\Dt}{\frac{\partial}{\partial t}}
\newcommand{\aBa}{(\alpha+1)B}
\newcommand{\GF}{\psi^{1+\frac{1}{2\alpha}}}
\newcommand{\GS}{\psi^{\frac12}}
\newcommand{\HFF}{\frac{\psi}{\rho}}
\newcommand{\HSS}{\frac{\psi}{\rho}}
\newcommand{\HFS}{\rho\psi^{\frac12-\frac{1}{2\alpha}}}
\newcommand{\HSF}{\frac{\psi^{\frac32+\frac{1}{2\alpha}}}{\rho}}
\newcommand{\AF}{\rho}
\newcommand{\AR}{\rho{\psi}^{\frac{1}{2}+\frac{1}{2\alpha}}}
\newcommand{\PF}{\alpha\frac{\psi}{|x|}}
\newcommand{\PS}{\alpha\frac{\psi}{\rho}}
\newcommand{\ds}{\displaystyle}
\newcommand{\Zt}{{\mathcal Z}^{t}}
\newcommand{\XPSI}{2\alpha\psi \begin{pmatrix} \frac{x}{\left< x \right>^2}\\ 0 \end{pmatrix} - 2\alpha\frac{{\psi}^2}{\rho^2}\begin{pmatrix} x \\ (\alpha +1)|x|^{-\alpha}y \end{pmatrix}}
\newcommand{\Z}{ \begin{pmatrix} x \\ (\alpha + 1)|x|^{-\alpha}y \end{pmatrix} }
\newcommand{\ZZ}{ \begin{pmatrix} xx^{t} & (\alpha + 1)|x|^{-\alpha}x y^{t}\\
     (\alpha + 1)|x|^{-\alpha}x^{t} y &   (\alpha + 1)^2  |x|^{-2\alpha}yy^{t}\end{pmatrix}}
\newcommand{\norm}[1]{\lVert#1 \rVert}
\newcommand{\ve}{\varepsilon}
\newcommand{\D}{\operatorname{div}}
\newcommand{\G}{\mathscr{G}}

\title[Space-like strong unique continuation, etc.]{ Space-like strong unique continuation for  some fractional parabolic equations}

\author{Vedansh Arya}
\address{Tata Institute of Fundamental Research\\
Centre For Applicable Mathematics \\ Bangalore-560065, India}\email[Vedansh Arya]{vedansh@tifrbng.res.in}

\author{Agnid Banerjee}
\address{Tata Institute of Fundamental Research\\
Centre For Applicable Mathematics \\ Bangalore-560065, India}\email[Agnid Banerjee]{agnidban@gmail.com}

\author{Donatella Danielli}
\address{Department of Mathematics\\ Arizona State University}\email[Donatella Danielli]{DDanielli@asu.edu}

\author{Nicola Garofalo}
\address{University of Padova, Italy}\email[Nicola Garofalo]{nicola.garofalo@unipd.it}

\thanks{A. Banerjee  is supported in part by SERB Matrix grant MTR/2018/000267 and by Department of Atomic Energy,  Government of India, under
project no.  12-R \& D-TFR-5.01-0520. N. Garofalo is supported in part by a Progetto SID: ``Non-local Sobolev and isoperimetric inequalities", University of Padova, 2019.}

%
%
%
\keywords{}
\subjclass{35A02, 35B60, 35K05}

\maketitle
\begin{abstract}
In this paper we establish the \emph{space-like} strong unique continuation for nonlocal equations of the type  $(\partial_t - \Delta)^s u= Vu$,  for $0<s <1$.  The proof of our  main result, Theorem \ref{main}, is achieved via a conditional elliptic type doubling property for solutions to the appropriate extension problem, followed by a blowup analysis.   
\end{abstract}

\tableofcontents
 
\section{Introduction and the statement of the main result}

The unique continuation property is one of the most fundamental aspects of the theory of partial differential equations. Its strong formulation states that if a solution to a certain PDE vanishes to infinite order at a point of a connected open set, then it must vanish identically in that set. Since harmonic functions are real analytic, it follows that the standard Laplacian has this property. However, it is a very non trivial fact that zero order perturbations of the Laplacian with, say, bounded potentials continue to have the strong unique continuation property. For parabolic equations the situation is much more delicate. An example of Frank Jones in \cite{Jr} shows that, in fact, there exist nontrivial solutions of the heat equation in $\R^{n+1}$  which are supported in a strip $\Rn \times (t_1, t_2)$. Although such solutions fail to satisfy the growth conditions at infinity that guarantee uniqueness for the heat equation, if we restrict Jones' example to a finite cylinder we infer that space-time backward propagation of zeros of infinite order fails for local solutions of the heat equation. 

In recent years there has been a considerable progress in the study of nonlocal operators following the seminal work \cite{CS} of Caffarelli and Silvestre. For an overview of the subject we refer the reader to e.g. \cite{Gft}. Despite such developments, the unique continuation property for the basic nonlocal model operators and their perturbations presently remains an area with many remarkable challenges. In this paper we settle the following  basic question concerning a class of nonlocal parabolic equations. In  $\Rn_x\times \R_t$ we consider the heat operator $H = \p_t -\Delta_x$ and denote by $H^s$ its fractional power of order $s\in (0,1)$. Our primary objective is to study \emph{space-like} strong unique continuation for solutions to the following nonlocal parabolic equation in a finite cylinder $B_1 \times (-1, 0]\subset \Rn_x\times \R_t$
\begin{equation}\label{e0}
H^s u(x,t) = - V(x,t) u(x,t),
\end{equation}
where on the potential $V$ we assume that for some $K>0$ one has 
\begin{equation}\label{vassump}
||V||_{C^1(B_1 \times (-1, 0])} \leq K, \ \text{if}\ s \in [1/2, 1),\ \ \ \ \ \ \ ||V||_{C^2(B_1 \times (-1, 0])} \leq K,\ \text{for}\ s \in (0,1/2).
\end{equation}
More precisely, we study functions $u \in \operatorname{Dom}(H^s) = \{u\in L^2(\Rn \times \R)\mid H^s u \in L^2(\Rn \times \R)\}$ that  satisfy \eqref{e0} above, and \emph{vanish to infinite order} at $(0,0)$ in the following sense: for any $k>0$, there exists $C_k>0$ such that
\begin{equation}\label{van1}
|u(x, t) | \leq C_k (|x |^2 + |t|)^{k/2}.
\end{equation}
Our main result is the following.

\begin{thrm}[Space-like strong unique continuation]\label{main}
Let $u \in \operatorname{Dom}(H^s)$ solve \eqref{e0} in $B_1 \times (-1, 0]$. Assume that $u$ vanishes  to infinite order at $(0,0)$ in the sense of \eqref{van1} above. Then it must be $u(\cdot, 0) \equiv 0$ in $\Rn$. 
\end{thrm}

The reader should keep in mind that, as a consequence of \cite[Theorem 5.1]{BG}, any solution to \eqref{e0} is continuous in $B_r \times (-r^2,0]$ for any $0<r\le 1$ . Therefore, $u(\cdot,0)$ makes sense classically. It is also worth re-emphasising here that, in view of Jones' cited example, the space-like propagation of zeros of infinite order claimed in Theorem \ref{main} is best possible, even for solutions to parabolic PDEs (i.e., in the local case when $s\nearrow 1$). In this connection we mention that for local solutions to second order parabolic equations space-like strong unique continuation results were proven in the remarkable works \cite{EF, EFV}. Our Theorem \ref{main} should be seen as the nonlocal counterpart of the ones in those papers. We also note that, although the equation \eqref{e0}  holds only in the finite cylinder $B_1 \times (-1, 0]$, the nonlocal nature of the operator $H^s$ forces $u(\cdot, 0)$ to vanish even outside $B_1$, where the equation  is not necessarily valid. This aspect is in the same  spirit of what happens in the elliptic case, see \cite{FF, Ru}. 

To provide some further perspective, we mention that for global solutions of the nonlocal equation \eqref{e0} a backward space-time strong unique continuation theorem was previously established by two of us in \cite{BG}. Such result represented the nonlocal counterpart of the one first obtained  by Poon in \cite{Po} for the local case $s=1$. In contrast with \cite{BG}, in this paper we do not assume that the equation \eqref{e0} hold globally.  This is the key new element of  the present work and it accounts for the different novel challenges, which we next describe.

The approach in \cite{BG} combined a monotonicity formula for an adjusted frequency function for the extension problem for $H^s$ (for the latter see \eqref{la} below)
with a  delicate blowup analysis  with respect to the so-called Almgren type rescalings. 
For Theorem \ref{main} above the situation is complicated by the fact that, since $u$ solves \eqref{e0} only locally, in view of the above comments the propagation of the zeros of $u$ is only expected to happen at $t=0$, but not necessarily backward in time.  Therefore,  because of the space-like thin propagation of zeros in the present scenario,   a blowup procedure  which is akin to that of \cite{FF, Ru, BG} is expected to encounter a  compactness obstruction. This aspect makes the analysis of the space-like strong unique continuation in the nonlocal parabolic framework somewhat subtle and different from both the elliptic case in \cite{FF, Ru}, and that of the strong backward uniqueness studied in \cite{BG}.

In order to overcome  such  hindrances we have taken inspiration from the beautiful ideas in  \cite{EFV}. However, as it is evident by looking at the work in Section \ref{s:main}, the adaption of their approach to our nonlocal setting has involved a novel delicate analysis. Roughly speaking, our approach is to argue by contradiction by first assuming that the solution of a certain parabolic extension problem \eqref{exprob} does not vanish in the normal direction $y$ at $t=0$. In this scenario, by adapting the frequency function approach introduced in \cite{GL}, \cite{GL2}, and subsequently extended in \cite{Po} to the heat equation, we establish an elliptic-type doubling property for such solutions similar to that in \cite{EFV}.  Such a conditional doubling property is the critical tool that allows us to effectively perform the blowup analysis. We eventually reduce the problem to a weak unique continuation property where the blowup limit solves the extension problem \eqref{exprob}, and also  vanishes identically both  in the tangential direction $x\in \Rn$ and in time.  The idea of establishing a conditional space-time doubling property for the extension problem is the key novelty of the present work. 

The paper is organised as follows. In Section \ref{s:n}, we introduce some basic notations and gather some preliminary results that are relevant to our work. In Section \ref{s:main} we establish the key results, Lemmas \ref{mont}, \ref{L:Db} and Theorem \ref{db1}, that constitute the novel part of our work, and its backbone. Once these auxiliary results are available, the proof of the main Theorem \ref{main} follows by the contradiction argument outlined above. 

As a closing comment we mention that the $C^2$ assumption on the potential $V$ in \eqref{vassump} enters indirectly in Theorem \ref{main}, in the sense that it is only needed for the auxiliary regularity Lemma \ref{reg1} below. The $C^2$ hypothesis is otherwise not directly used in the work in this paper.

\section{Notations and Preliminaries}\label{s:n}

In this section we introduce the relevant notation and gather some auxiliary  results that will be useful in the rest of the paper. Generic points in $\Rn \times \R$ will be denoted by $(x_0, t_0), (x,t)$, etc. For an open set $\Omega\subset \Rn_x\times \R_t$ we indicate with $C_0^{\infty}(\Omega)$ the set of compactly supported smooth functions in $\Om$. We also indicate by $H^{\alpha}(\Omega)$ the non-isotropic parabolic H\"older space with exponent $\alpha$ defined in \cite[p. 46]{Li}. The symbol $\mathscr S(\R^{n+1})$ will denote the Schwartz space of rapidly decreasing functions in $\R^{n+1}$. For $f\in \mathscr S(\R^{n+1})$ we denote its Fourier transform by 
\[
\mathscr F_{(x,t)\to (\xi,\sigma)}(f) = \hat f(\xi,\sigma) = \int_{\Rn\times \R} e^{-2\pi i(\langle \xi,x\rangle + \sigma t)} f(x,t) dx dt = \mathscr F_{x\to\xi}(\mathscr F_{t\to\sigma} f).
\]
The heat operator in $\R^{n+1} = \Rn_x \times \R_t$ will be denoted by $H = \p_t - \Delta_x$. Given a number $s\in (0,1)$ the notation $H^s$ will indicate the fractional power of $H$ that in  \cite[formula (2.1)]{Sam} was defined on a function $f\in \mathscr S(\R^{n+1})$ by the formula
\begin{equation}\label{sHft}
\widehat{H^s f}(\xi,\sigma) = (4\pi^2 |\xi|^2 + 2\pi i \sigma)^s\  \hat f(\xi,\sigma),
\end{equation}
with the understanding that we have chosen the principal branch of the complex function $z\to z^s$. If we denote parabolic dilations with $\delta_\la(x,t) = (\la x,\la^2 t)$, then one easily recognises from \eqref{sHft} that
\begin{align*}
& H^s(\delta_\la f)(\xi,\sigma) = \la^{2s} H^s f(\la \xi,\la^2 \sigma),
\end{align*}
which shows that $H^s$ is a pseudodifferential operator of parabolic homogeneity $2s$. This motivates the introduction of the parabolic Sobolev space of fractional order $2s$ that represents the natural framework for the present work 
\begin{align}\label{dom}
\mathscr H^{2s} & =  \operatorname{Dom}(H^s)   = \{f\in \mathscr S'(\R^{n+1})\mid f, H^s f \in L^2(\R^{n+1})\}
\\
&  = \{f\in L^2(\R^{n+1})\mid (\xi,\sigma) \to (4\pi^2 |\xi|^2 + 2\pi i \sigma)^s  \hat f(\xi,\sigma)\in L^2(\R^{n+1})\},
\notag
\end{align} 
where the second equality is justified by \eqref{sHft} and Plancherel theorem. 
For the purpose of \eqref{la} below it is important to keep in mind that definition \eqref{sHft} is equivalent to the one based on Balakrishnan formula (see \cite[(9.63) on p. 285]{Samko})
\begin{equation}\label{balah}
H^s f(x,t) = - \frac{s}{\Gamma(1-s)} \int_0^\infty \frac{1}{\tau^{1+s}} \big(P^H_\tau f(x,t) - f(x,t)\big) d\tau,
\end{equation}
where we have denoted by
\begin{equation}\label{evolutivesemi}
P^H_\tau f(x,t) = \int_{\Rn} G(x-y,\tau) f(y,t-\tau) dy = G(\cdot,\tau) \star f(\cdot,t-\tau)(x)
\end{equation}
the \emph{evolutive semigroup}, see \cite[(9.58) on p. 284]{Samko}, i.e., the solution $u((x,t),\tau) = 
P^H_\tau f(x,t)$ of the Cauchy problem in $\R^{n+1}_{(x,t)}\times \R^+_\tau$
\[
\p_\tau u = \Delta_x u - \p_t u,\ \ \ \ \ \ u((x,t),0) = f(x,t).
\]
It is easy to recognise that \eqref{balah} is equivalent to \eqref{sHft}. For that, one observes that \eqref{evolutivesemi} gives
\begin{align*}
& \widehat{P^H_\tau f}(\xi,\sigma) = \mathscr F_{x\to\xi}(G(\cdot,\tau) \star \mathscr F_{t\to\sigma} f(\cdot,\cdot-\tau))
\\
& = \mathscr F_{x\to\xi}(G(\cdot,\tau) \star e^{-2\pi i\sigma \tau} \mathscr F_{t\to\sigma} f(\cdot,\cdot)) = e^{-4\pi^2 \tau |\xi|^2} e^{-2\pi i\sigma \tau} \hat f(\xi,\sigma)
\\
& = e^{-\tau(4\pi^2 |\xi|^2 + 2\pi i\sigma)} \hat f(\xi,\sigma).
\end{align*}
Using this identity and taking Fourier transform in \eqref{balah} we thus find for any $(\xi,\sigma)\in \R^{n+1}$ with $\xi\not=0$, one has
\[
\widehat{H^s f}(\xi,\sigma) = - \frac{s}{\Gamma(1-s)} \int_0^\infty \frac{1}{\tau^{1+s}} \big(e^{-\tau(4\pi^2 |\xi|^2 + 2\pi i\sigma)} -1\big)  d\tau\ \hat f(\xi,\sigma).
\]
Applying to the latter formula the well-known identity 
\[
- \frac{s}{\Gamma(1-s)} \int_0^\infty \frac{e^{-t L} - 1}{t^{1+s}} dt = L^s,
\]
valid for every $0<s<1$  and for $L\in \mathbb C$ such that $\Re L>0$, we see that \eqref{balah} implies \eqref{sHft}. The proof that \eqref{sHft} implies \eqref{balah} follows by back tracing the above steps.

Henceforth, given a point $(x,t)\in \R^{n+1}$ we will consider the thick half-space $\R^{n+1} \times \R^+_y$. At times it will be convenient to combine the additional variable $y>0$ with $x\in \Rn$ and denote the generic point in the thick space $\Rn_x\times\R^+_y$ with the letter $X=(x,y)$. For $x_0\in \Rn$ and $r>0$ we let $B_r(x_0) = \{x\in \Rn\mid |x-x_0|<r\}$,
$\mathbb B_r(x_0,0)=\{X = (x,y) \in \R^n \times \R^{+}\mid |x-x_0|^2 + y^2 < r^2\}$ (note that this is the upper half-ball),
and $\mathbb Q_r((x_0,t_0),0)=\mathbb B_r(x_0,0) \times (t_0,t_0+r^2]$.
When the center $x_0$ of $B_r(x_0)$ is not explicitly indicated, then we are taking $x_0 = 0$. Similar agreement for the thick half-balls $\mathbb B_r(x_0,0)$. For notational ease $\nabla U$ and  $\operatorname{div} U$ will respectively refer to the quantities  $\nabla_X U$ and $ \operatorname{div}_X U$.  The partial derivative in $t$ will be denoted by $\p_t U$ and also at times  by $U_t$. The partial derivative $\partial_{x_i} U$  will be denoted by $U_i$. At times,  the partial derivative $\partial_{y} U$  will be denoted by $U_{n+1}$. 

We next introduce a boundary value problem that represents an essential tool when working with $H^s$. Given a number $a\in (-1,1)$ and a $u:\R^n_x\times \R_t\to \R$ we seek a function $U:\R^n_x\times\R_t\times \R_y^+\to \R$ that satisfies the boundary-value problem
\begin{equation}\label{la}
\begin{cases}
\La U \overset{def}{=} \partial_t (y^a U) - \operatorname{div} (y^a \nabla U) = 0,
\\
U((x,t),0) = u(x,t),\ \ \ \ \ \ \ \ \ \ \ (x,t)\in \R^{n+1}.
\end{cases}
\end{equation}
The most basic property of the Dirichlet problem \eqref{la} is that if $s = \frac{1-a}2\in (0,1)$, then for any $1\le p\le \infty$ one has in $L^p(\R^{n+1})$
\begin{equation}\label{np}
2^{-a}\frac{\Gamma(\frac{1-a}{2})}{\Gamma(\frac{1+a}{2})} \py U((x,t),0)=  - H^s u(x,t),
\end{equation}
where $\py$ denotes the weighted normal derivative
\begin{equation}\label{nder}
\py U((x,t),0)\overset{def}{=}   \operatorname{lim}_{y \to 0^+}  y^a \partial_y U((x,t),y).
\end{equation}
We note explicitly that the sign convention in \eqref{nder} corresponds to the opposite of that of the outer normal derivative on $\p \R_+^{n+1} = \Rn_x\times\{0\}_y$.
When $a = 0$ ($s = 1/2$) the problem \eqref{la} was first introduced in \cite{Jr1} by Frank Jones, who in such case also constructed the relevant Poisson kernel and proved \eqref{np}. More recently Nystr\"om and Sande in \cite{NS} and Stinga and Torrea in \cite{ST} have independently extended the results in \cite{Jr} to all $a\in (-1,1)$. 

With this being said, we now suppose that $u$ be a solution to \eqref{e0} with $V$ satisfying \eqref{vassump}, and consider the weak solution $U$ of the following version of \eqref{la} (for the precise notion of weak solution of \eqref{wk} we refer to \cite[Section 4]{BG}) 
\begin{equation}\label{wk}
\begin{cases}
\La U=0 \ \ \ \ \ \ \ \ \ \ \ \ \ \ \ \ \ \ \ \ \ \ \ \ \ \ \ \ \ \ \text{in}\ \R^{n+1}\times \R^+_y,
\\
U((x,t),0)= u(x,t)\ \ \ \ \ \ \ \ \ \ \ \ \ \ \ \ \text{for}\ (x,t)\in \R^{n+1},
\\
\py U((x,t),0)=  2^{a} \frac{\Gamma(\frac{1+a}{2})}{\Gamma(\frac{1-a}{2})} V(x,t) u(x,t)\ \ \ \ \text{for}\ (x,t)\in B_1 \times (-1,0].
\end{cases}
\end{equation}
 Note that the third equation in \eqref{wk} is justified by \eqref{e0} and \eqref{np}. In \cite[Corollary 4.6]{BG} two of us have shown that if $u\in \mathscr H^{2s}$ (see \eqref{dom} above), then the extended function $U$ constructed as in \eqref{la} is a weak solution of \eqref{wk}. Further, in \cite[Lemma 5.3]{BG} the following regularity result for such weak solutions was proved. Such result will play a pervasive role in our analysis.

\begin{lemma}\label{reg1}
Let $U$  be a weak solution of \eqref{wk}. Then there exists $\alpha>0$ such that one has up to the thin set $\{y=0\}$ 
\[
U_i,\ U_t,\ y^a U_y\ \in\ H^{\alpha}(\mathbb B_\frac{1}{2} \times (-\frac 14, 0]),\ \ \ \  i=1,2,..,n.
\]
Moreover, the relevant H\"older norms are bounded by $\int_{\mathbb B_1 \times (-1, 0)} U^2 y^a dX dt$. 
\end{lemma}  
We next recall that it was shown in \cite{Gcm} that given $\phi \in C_0^{\infty}(\R^{n+1}_+)$ the solution of the Cauchy problem with Neumann condition
\begin{align}\label{CN}
\begin{cases}
\La U=0 \hspace{2mm} &\text{in}  \hspace{2mm}\R^{n+1}_+ \times (0,\infty)\\
U(X,0)=\phi(X), \hspace{2mm} &X \in \R^{n+1}_+,\\
\py U(x, 0, t) = 0 \hspace{2mm} & x\in \Rn,\ t \in (0, \infty)
\end{cases}
\end{align}
is given by the formula 
\begin{equation}\label{Ga}
\mathscr P^{(a)}_t \phi(X_1) \overset{def}{=} U(X_1,t) = \int_{\R^{n+1}_+} \phi(X) \G (X_1,X,t) y^a dX,
\end{equation}
where 
\begin{equation}\label{funda}
\G(X_1,X,t) = p(x_1,x,t)\ p^{(a)}(y_1,y,t)
\end{equation}
is the product of the standard Gauss-Weierstrass kernel $p(x_1,x,t) = (4\pi t)^{-\frac n2} e^{-\frac{|x_1-x|^2}{4t}}$ in $\Rn\times\R^+$ with the heat kernel of the Bessel operator $\mathscr B_a = \p_{yy} + \frac ay \p_y$  with Neumann boundary condition in $y=0$ on $(\R^+,y^a dy)$ (reflected Brownian motion)
 \begin{align}\label{fs}
p^{(a)}(y_1,y,t) & =(2t)^{-\frac{a+1}{2}}\left(\frac{y_1 y}{2t}\right)^{\frac{1-a}{2}}I_{\frac{a-1}{2}}\left(\frac{y_1 y}{2t}\right)e^{-\frac{y_1^2+y^2}{4t}}.
\end{align}
In \eqref{fs} we have denoted by $I_{\frac{a-1}{2}}$ the modified Bessel function of the first kind and order $\frac{a-1}{2}$ defined by the series 
\begin{align}\label{besseries}
 I_{\frac{a-1}{2}}(z) = \sum_{k=0}^{\infty}\frac{(z/2)^{\frac{a-1}{2}+2k} }{\Gamma (k+1) \Gamma(k+1+(a-1)/2)}, \hspace{4mm} |z| < \infty,\; |\operatorname{arg} z| < \pi.
\end{align}
For future use we note explicitly that \eqref{funda} and \eqref{fs} imply that for every $x, x_1\in \Rn$ and $t>0$ one has
\begin{equation}\label{neu}
\operatorname{lim}_{y \to 0^+}  y^a \partial_y \G((x,y),(x_1,0),t) =0. 
\end{equation}
It is important for the reader to keep in mind that \eqref{Ga} defines a stochastically complete semigroup (see \cite[Propositions 2.3 and 2.4]{Gcm}), and therefore in particular we have for every $X_1\in \R^{n+1}_+$ and $t>0$ 
\begin{equation}\label{Ga1}
\mathscr P^{(a)}_t 1(X_1) = \int_{\R^{n+1}_+} \G (X_1,X,t) y^a dX = 1,
\end{equation}
and also $\mathscr P^{(a)}_t \phi(X_1) \underset{t\to 0^+}{\longrightarrow} \phi(X_1)$.

From the asymptotic behaviour of $I_{\frac{a-1}{2}}(z)$ near $z=0$ and at infinity one immediately obtains the following estimate for some $C(a), c(a) >0$ (see e.g. \cite[formulas (5.7.1) and (5.11.8)]{Le}) ,
\begin{equation}\label{bessel}
I_{\frac{a-1}{2}}(z) \leq C(a) z^{\frac{a-1}{2}}  \hspace{6mm} \text{if} \hspace{2mm} 0 < z \le c(a),\ \ \ \ \ I_{\frac{a-1}{2}}(z) \leq C(a) z^{-1/2} e^z \hspace{2mm}\ \  \text{if} \hspace{2mm} z \ge c(a).
\end{equation}
We also note explicitly that \eqref{besseries} gives as $z\to 0^+$
\begin{equation}\label{zero}
z^{\frac{1-a}{2}} I_{\frac{a-1}{2}}(z) \cong \frac{2^{\frac{1-a}{2}}}{\Gamma((1+a)/2)}.
\end{equation}

 We now state three auxiliary results that will be needed in our work. 

\begin{lemma}[Hardy type inequality]\label{escves3}
For all $h \in C_0^{\infty}(\overline{\R^{n+1}_+})$ and $b>0$ the following inequality holds
\begin{align*}
& \int_{\R^{n+1}_+} h^2 \frac{|X|^2}{8b} e^{-|X|^2/4b}y^adX \leq  2b \int_{\R^{n+1}_+} |\n h|^2 e^{-|X|^2/4b}y^a dX 
\\
& + \frac{n+1+a}{2} \int_{\R^{n+1}_+} h^2  e^{-|X|^2/4b} y^a dX.
\end{align*}
\end{lemma}

\begin{proof}
We observe that since $\nabla(e^{-|X|^2/4b}) = - \frac{X}{2b} e^{-|X|^2/4b}$, we can write
\begin{align*}
& \int_{\R^{n+1}_+} h^2 \frac{|X|^2}{8b} e^{-|X|^2/4b}y^adX 
 =-\frac{1}{4} \int_{\R^{n+1}_+} h^2 \langle X, \nabla_X e^{-|X|^2/4b}\rangle  y^a dX.\end{align*}
Keeping in mind that by Euler formula $\langle X,\nabla_X y^a\rangle = a y^a$, that $\operatorname{div}_X X = n+1$, and that denoting with $\nu$ the outer unit normal to a sufficiently large half-ball $\mathbb B_r$ we have $\langle X,\nu\rangle = 0$ on $\p \mathbb B_r \cap \{y=0\}$, integrating by parts in the latter identity we find
\[
\int_{\R^{n+1}_+} h^2 \frac{|X|^2}{8b} e^{-|X|^2/4b}y^adX= \frac{n+1+a}{4} \int_{\R^{n+1}_+}  h^2 e^{-|X|^2/4b}  y^a dX
 + \frac{1}{2}\int_{\R^{n+1}_+} h \langle\nabla_X h, X\rangle e^{-|X|^2/4b} y^a dX.
\]
Cauchy-Schwarz inequality now gives
\[
\frac{1}{2}\int_{\R^{n+1}_+} h \langle\nabla_X h, X\rangle e^{-|X|^2/4b} y^a dX \leq \int_{\R^{n+1}_+} \frac{|X|^2}{16 b} h^2 e^{-|X|^2/4b} y^a dX + b \int_{\R^{n+1}_+} |\nabla h|^2 e^{-|X|^2/4b} y^a dX.
\]
Substituting in the above equality we reach the desired conclusion.
 
\end{proof}

The next result will be needed in the proof of (i) in Theorem \ref{db1}. Its proof relies on Lemma \ref{escves3} and is completely analogous to that of \cite[Lemma 4]{EFV} for the case $a=0$. We therefore refer to that source and omit the relevant details.
 
\begin{lemma}\label{escves4}
Let $h \in C_0^{\infty}(\overline{\R^{n+1}_+})$. Assume $N$ and $\theta$ verify $N \operatorname{log}(N \theta) \geq 1$ and that the inequality 
\begin{align*}
2b \int_{\R^{n+1}_+} |\n h|^2 e^{-|X|^2/4b}y^a dX + \frac{n+1+a}{2} \int_{\R^{n+1}_+} h^2  e^{-|X|^2/4b} y^a dX \leq N \operatorname{log}(N \theta) \int_{\R^{n+1}_+} h^2  e^{-|X|^2/4b} y^a dX
	\end{align*}
hold when $0<b \leq 1/N \operatorname{log}(N \theta)$. Then for all $r \leq \frac{1}{2}$ one has 
\begin{align*}
	\int_{\mathbb B_{2r}} h^2 y^a dX \leq (N \theta)^N \int_{\mathbb B_{r}} h^2 y^a dX.
\end{align*}
\end{lemma}

We will also repeatedly use the following time-independent trace inequality. For its proof see  \cite[p. 65]{Ru}.

\begin{lemma}[Trace inequality]\label{tr}
Let   $f\in C_0^\infty(\R^{n+1}_+)$. There exists a constant $C_0 = C_0(n,a)>0$ such that for every $\sigma>1$ one has
\[
\int_{\Rn} f(x,0)^2 dx \le C_0 \left(\sigma^{1+a} \int_{\R^{n+1}_+} f(X)^2 y^a dX + \sigma^{a-1} \int_{\R^{n+1}_+} |\nabla f(X)|^2 y^a dX\right).
\]
\end{lemma}
In the final step of the proof of Theorem \ref{main}, when we analyse the blowup limit, we will need the following  weak unique continuation result from \cite[Proposition 5.6]{LLR}. The reader should be aware that we are stating their result for the backward equation since this is the form that we use. One obtains one from the other by simply changing $t$ into $-t$.

\begin{prop}\label{wucp}
Let $U_0$ be a  weak solution to
\begin{equation}\label{homog}
\begin{cases}
y^a \partial_t U_0 + \D(y^a \nabla U_0) =0\ \text{ in  $\mathbb B_1 \times [0,1)$,}
\\
\py U_0((x,0), t) \equiv 0\ \text{for all $(x,t)  \in B_1 \times [0,1)$,}
\end{cases}
\end{equation}
such that $U_0((x,0), t) \equiv 0$ for all $(x,t) \in B_1 \times [0,1)$. 
Then $U_0 \equiv 0$ in $\mathbb B_1 \times [0,1)$.
\end{prop}


\section{Proof of Theorem \ref{main}}\label{s:main}

In this section we prove our main result, Theorem \ref{main}. For notational purposes it will be convenient to work with the following backward version of problem \eqref{wk} in the cylinder $\mathbb Q_4 = \mathbb B_4 \times (0,16]$
\begin{equation}\label{exprob}
\begin{cases}
y^a \partial_t U + y^a \operatorname{div}(y^a \nabla U)=0\ \ \ \ \ \ \ \ \ \ \ \ \text{in} \ \mathbb Q_4,
\\	
U((x,t),0)= u(x,t)
\\
\py U((x,t),0)= V(x,t) u(x,t)\ \ \ \ \ \ \ \ \ \text{in}\ B_4 \times [0,16).
\end{cases}
\end{equation}
We note that the former can be transformed into the latter by changing $t \to -t$ and a parabolic rescaling $U_{r_0}(x,t)= U(r_0x, r_0^2t)$, for small enough $r_0$. We also emphasise that, for the sake of simplifying the notation, in \eqref{exprob} we have incorporated in the potential $V$ the normalising constant $2^{a} \frac{\Gamma(\frac{1+a}{2})}{\Gamma(\frac{1-a}{2})}$ in \eqref{wk}. 
Before proceeding we furthermore alert the reader that, using the regularity result in Lemma \ref{reg1} above, similarly to what was done in \cite[Section 6]{BG}, the computations in the ensuing Lemmas \ref{mont} and \ref{db1} can be rigorously justified by first considering integrals in the region $\{y>\ve\}$ and then letting $\ve \to 0$. 

Our first  lemma is a  monotonicity in time result which in our context  finally allows the passage of information to $t=0$ for the extension problem \eqref{exprob}.  This is akin to \cite[Lemma 1]{EFV}. However, as the reader will see, the proof of this lemma is  somewhat subtler in our situation since the extension operator is not translation invariant in the $y$ variable and therefore one requires a careful analysis  using the  bounds  on the  fundamental solution of the extension PDE in order to obtain the required estimates. Moreover our situation is also somewhat complicated by the presence of the weighted Neumann condition. 

With this being said, we now introduce an assumption that will remain in force for the rest of the section up to the proof of Theorem \ref{main}. When we work with a solution $U$ of the problem \eqref{exprob}, we will always assume that
\begin{equation}\label{ass}  
\int_{\mathbb B_1} U(X,0)^2 y^adX >0.
\end{equation}
As a consequence of such hypothesis the number 
\begin{equation}\label{theta}
\theta \overset{def}{=}\frac{\int_{\mathbb Q_4} U(X,t)^2 y^adXdt }{\int_{\mathbb B_1} U(X,0)^2 y^adX}
\end{equation}
will be well-defined. In the remainder of this work the symbol $\theta$ will always mean the number defined by \eqref{theta}. 
\begin{lemma}\label{mont} 
Let $U$ be a solution of \eqref{exprob} in $\mathbb Q_4$. Then there exists a constant $N = N(n,a,||V||_1)>2$ such that $N\operatorname{log}(N\theta) \geq 1$, and for which the following inequality holds for $0\leq t \leq 1/{N\operatorname{log}(N\theta)}$,
\begin{align*}
N\int_{\mathbb B_2} U(X,t)^2 y^adX \geq \int_{\mathbb B_1} U(X,0)^2 y^adX.
\end{align*}	
\end{lemma}
\begin{proof}
Let  $f= \phi U,$ where $\phi \in C_0^{\infty}(\mathbb B_2)$ is a spherically symmetric cutoff such that $0\le \phi\le 1$ and $\phi \equiv1$ on $\mathbb B_{3/2}.$ Since $U$ solves \eqref{exprob} and $\phi$ is independent of $t$, it is easily seen that the function $f$ solves the problem
\begin{equation}\label{feq}
\begin{cases}
y^a f_t + \D(y^a \nabla f) = 2 y^a \langle\nabla U,\nabla \phi\rangle  + \D(y^a \nabla \phi) U\ \ \ \ \ \ \ \text{in} \ \mathbb Q_4,
\\	
f((x,t),0)= u(x,t)\phi(x,0)
\\
\py f((x,t),0)= V(x,t) f(x,t)\ \ \ \ \ \ \ \ \ \ \ \ \ \ \ \ \ \ \ \ \ \ \ \ \ \ \ \ \ \ \ \ \ \text{in}\ B_4 \times [0,16).
\end{cases}
\end{equation}
Since $\phi$ is symmetric in $y$, we have $\phi_y \equiv 0$ on the thin set $\{y=0\}$. This fact and the smoothness of $\phi$ imply that $\frac{\phi_y}{y}$ be bounded up to $\{y=0\}$. Therefore we observe that the following is true 
\begin{equation}\label{obs1}
\begin{cases}
\operatorname{supp} (\nabla \phi) \cap \{y>0\}  \subset \mathbb B_2 \setminus \mathbb B_{3/2}
\\
|\D(y^a \nabla \phi)| \leq C y^a\ \mathbf 1_{\mathbb B_2 \setminus \mathbb B_{3/2}},
\end{cases}
\end{equation}
where for a set $E$ we have denoted by $\mathbf 1_E$ its indicator function.

To establish the lemma we now fix a point $X_1=(x_1,y_1)\in \R^{n+1}_+$ and introduce the quantity
\begin{align*}
H(t) = \int_{\R^{n+1}_+} f(X,t)^2 \G(X_1,X,t) y^a dX = \mathscr P_t^{(a)}(f(\cdot,t)^2)(X_1).
\end{align*}
We observe explicitly that by the semigroup property of $\mathscr P_t^{(a)}$ and the fact that $\phi \equiv 1$ in $\mathbb B_1$, we have for every $X_1\in \mathbb B_1$
\begin{equation}\label{H0}
\underset{t\to 0^+}{\lim}\ H(t) = f(X_1,0)^2 = U(X_1,0)^2. 
\end{equation}
We intend to establish the following. 

\medskip

\noindent \emph{\underline{Claim}}: There exist constants $C(n,a), C(n,a,||V||_1)>0$ and $0<t_0 = t_0(n,a,||V||_{\infty})<1$ such that for $X_1\in \mathbb B_1$ and $0<t<t_0$ one has 
\begin{equation}\label{mkg1}
H'(t) \geq -C(n,a)||V||_{\infty} t^{-\frac{1+a}2} H(t) -C(n,a,||V||_1) e^{-\frac{1}{C(n,a)t}}||U||_{L^2(\mathbb Q_4, y^a dXdt)}^2.
\end{equation}

\medskip

Once the \emph{\underline{Claim}} is proved we can complete the proof of the lemma as follows. With 
\begin{equation}\label{C}
C=\max\{C(n,a)(||V||_{\infty}+1) +C(n,a,||V||_1), 1\},
\end{equation}
we easily obtain from \eqref{mkg1} for $0<t<t_0$ 
\[
\left[e^{\frac{2C}{1-a} t^{\frac{1-a}2}}H(t)\right]' \geq -C e^{\frac{2C}{1-a} t^{\frac{1-a}2}} e^{-\frac{1}{Ct}}||U||_{L^2(\mathbb Q_4, y^a dXdt)}^2.
\]
For a given $0<t<t_0$ we now integrate such differential inequality on the interval $(0,t)$. Keeping \eqref{H0} in mind and that $a<1$ we find for any fixed $X_1 \in \mathbb B_1$
\begin{equation}\label{N0H}
N_0 H(t) \geq U(X_1,0)^2 -N_0 e^{-\frac{1}{N_0 t}}||U||_{L^2(\mathbb Q_4, y^a dXdt)}^2,
\end{equation} 
where, if $t_0>0$ is as in the claim, we have let $N_0=C e^{\frac{2C}{1-a} t_0^{\frac{1-a}2}}$ with $C$ as in \eqref{C}. We note explicitly that $N_0\ge 1$.
Integrating now \eqref{N0H} with respect to $X_1\in \mathbb B_1$, exchanging the order of integration and using \eqref{Ga1}, we obtain for $t \leq t_0$
 \begin{align}\label{11}
 N_0\int_{\mathbb B_2}U(X,t)^2 y^a dX \geq \int_{\mathbb B_1}U(X,0)^2 y^a dX - N_0 e^{-\frac{1}{N_0 t}}||U||_{L^2(\mathbb Q_4, y^adXdt)}^2,
 \end{align}
Note that in \eqref{11} we have renamed for convenience  the variable $X_1$ as $X$. If we now use the $L^{\infty}$ bounds on $U$ in \cite[Theorem 5.1]{BG} we conclude that there exist a constant $C(n,a,||V||_1)>0$ such that
\begin{equation}\label{N0}
\int_{\mathbb B_1}U(X,0)^2 y^a dX  \leq  C(n,a,||V||_1)\int_{\mathbb Q_4}U(X,t)^2 y^adXdt.
\end{equation}
The estimate \eqref{N0} implies the crucial conclusion that $\theta$ is universally bound from below away from zero, i.e.  
\[
\theta = \frac{\int_{\mathbb Q_4}U(X,t)^2 y^adXdt}{\int_{\mathbb B_1}U(X,0)^2 y^a dX} \ge \frac{1}{C(n,a,||V||_1)} >0.
\]
Therefore, by possibly adjusting the choice of the constant $C$ in the above definition of $N_0$, depending on the constant $C(n,a,||V||_1)$, we can arrange that $ \operatorname{log}(N_0\theta) >1$. Having fixed $N_0$ in this way we now take $N=2N_0/t_0>2N_0$, the second inequality being justified since $0<t_0<1$ in the claim. With such choice of $N$ we claim that 
\[
0< t\leq 1/N \operatorname{log}(N\theta)\ \implies\ 0<t \leq t_0.
\] 
For this to be true it suffices to have $\frac{1}{N \operatorname{log}(N\theta)}\le t_0$, which is equivalent to $2 N_0 \log(\frac 2{t_0} N_0\theta) \ge 1$. Since $N_0\ge 1$, this is obviously true by our choice of $N_0$. Furthermore, it is also obvious that $N\log(N\theta)\ge 1$. Suppose finally that $0<t\le 1/N \operatorname{log}(N\theta)$. This trivially implies $0<t\le 1/N_0 \log(2N_0\theta)$ and therefore $N_0 e^{-\frac{1}{N_0 t}} \theta \le \frac 12$, which means
\[
N_0 e^{-\frac{1}{N_0 t}}||U||_{L^2(\mathbb Q_4,y^adX)}^2 \leq \frac{1}{2}\int_{\mathbb B_1}U(X,0)^2 y^adX.
\] 
Substituting this inequality in  \eqref{11} we thus obtain
  $$	2N_0\int_{\mathbb B_2}U(X,t)^2y^adX \geq \int_{\mathbb B_1}U(X,0)^2 y^adX,$$
which finally proves the lemma.

We are thus left with proving the \emph{\underline{Claim}}. Henceforth, we will routinely omit the domain of integration $\R^{n+1}_+$ in all integrals involved (the reader should keep in mind that the function $f$ is actually supported in the upper half-ball $\mathbb B_2\subset \R^{n+1}_+$). If instead the relevant integral is on $\p \R^{n+1}_+ = \Rn_x\times\{0\}_y$, then we simply write $\{y = 0\}$ for the domain of integration. Also, we will simply write $\G$ instead of $\G(X_1,X,t)$ in all integrals on $\R^{n+1}_+$. In instead the integral is on $\{y=0\}$, the notation $\G$ will stand for $\G(X_1,(x,0),t)$.
Using the divergence theorem, \eqref{feq} and the Neumann condition \eqref{neu} above we obtain for any fixed $X_1\in \mathbb B_1$ and $0<t\le 1$
\begin{align}\label{h1}
H'(t) &  = \int 2 f f_t \G y^a dX +\int f^2 \G_t y^a dX
= \int 2 f f_t \G y^a dX +\int f^2\D(y^a  \nabla\G)dX
\\
&= \int 2 f f_t \G y^a dX -\int \langle\nabla f^2,\nabla \G\rangle y^a dX
= \int 2 f f_t \G y^a dX +\int \D(y^a \nabla f^2) \G dX
\notag\\
& +\int_{\{y=0\}} 2V f^2 \G dx
\notag \\
&= \int 2 f (y^a f_t + \D(y^a \nabla f)) \G  dX
 + \int 2 |\nabla f|^2   \G y^a dX +\int_{\{y=0\}} 2Vf^2\G dx\notag\\
&= I_1 + I_2 + I_3.\notag
\end{align}
Our objective is to prove that:
\begin{itemize}
\item[(i)] for every $X_1\in \mathbb B_1$ and $0<t\le 1$ we have 
\begin{equation}\label{i1}
I_1 \geq - C e^{-\frac{1}{M t}} \int_{\mathbb Q_4} U^2 y^a dXdt. 
\end{equation}
\item[(ii)] there exists $t_0<1$ such that for every $X_1\in \mathbb B_1$ and $0<t\le t_0$ one has
\begin{align}\label{traceapp}
|I_3| \leq C(n,a)||V||_{\infty}\left(t^{-\frac{1+a}{2}}\int f^2 \G y^a dX + t^{\frac{1-a}{2}} \int |\nabla f|^2 \G y^a dX\right).
\end{align}
\end{itemize}
With \eqref{i1} and \eqref{traceapp} in hands, we return to \eqref{h1} to find
\begin{align*}
H'(t)  \ge & - C e^{-\frac{1}{M t}} \int_{\mathbb Q_4} U^2 y^a dXdt + 2\int  |\nabla f|^2   \G y^a dX
\\
& - C(n,a)||V||_{\infty}\left(t^{-\frac{1+a}{2}} H(t) + t^{\frac{1-a}{2}} \int |\nabla f|^2 \G y^a dX\right).
\end{align*}
If at this point in this inequality we choose $t_0<1$  such that $C(n,a)||V||_{\infty} t_0^{\frac{1-a}{2}} <1$, it is clear that  for $X_1\in \mathbb B_1$ and $0<t<t_0$ we obtain the \emph{Claim} \eqref{mkg1}.

To finally complete the proof of the lemma we are thus left with establishing \eqref{i1} and \eqref{traceapp}. With this objective in mind, with $c(a)$ as in \eqref{bessel} we first write $I_1 = I_1^1 + I_1^2,$ where $I_1^1$ is integral on the set  $ A=\{ X\in \R^{n+1}_+\mid yy_1 > 2 t c(a)\}$ and  $I_1^2$ is the integral on the complement $A^c$ of $A$. We want to bound $I_1$ by appropriately bounding $\G$ from above in each of the sets $A$ and $A^c$. In this respect it will be important for the reader to keep in mind that in view of \eqref{feq} and \eqref{obs1} the integral in the definition of $I_1$ is actually performed in $X\in \mathbb B_2 \setminus  \mathbb B_{3/2}$ and on such set we have for every $X_1\in \mathbb B_1$
\begin{equation}\label{X}
\frac 12 \le |X-X_1|\le 3.
\end{equation}
Our objective is to prove that when $X_1 \in \mathbb B_1$, $X \in \mathbb B_2 \setminus  \mathbb B_{3/2}$ and $0<t\le 1$, the following bound holds for some universal $M>0$ 
\begin{equation}\label{g4}
\G(X_1, X, t) \leq e^{-\frac{1}{M t}}.
\end{equation}
To prove that \eqref{g4} is true when $X\in A\cap(\mathbb B_2 \setminus  \mathbb B_{3/2})$ we argue as follows. Since for $X\in A$ we have $\frac{yy_1}{2t} >  c(a)$, by the second inequality in \eqref{bessel} we have
\begin{equation}\label{gud}
I_{\frac{a-1}{2}}\left(\frac{y_1 y}{2t}\right) \le C(a) \left(\frac{y_1 y}{2t}\right)^{-1/2} e^{\frac{y_1 y}{2t}}.
\end{equation}
Consider first the case $-1<a \le 0$. Since for $X\in \mathbb B_2$ and $X_1\in \mathbb B_1$ we trivially have $\frac{yy_1}{2t} \leq \frac{y_1^2 +y^2}{2t}\le \frac{4}t$, in such case we have $\left(\frac{y y_1}{2t}\right)^{-a/2} \leq  2^{-a/2}t^{a/2}$. Using this estimate and \eqref{gud} in \eqref{fs} we obtain
\begin{align*}
p^{(a)}(y_1,y,t) & \le 2^{-a/2}t^{a/2} (2t)^{-\frac{a+1}{2}}\left(\frac{y_1 y}{2t}\right)^{\frac{1}{2}}I_{\frac{a-1}{2}}\left(\frac{y_1 y}{2t}\right)e^{-\frac{y_1^2+y^2}{4t}}
\\
& \le C^\star(a) t^{-1/2} e^{-\frac{(y_1-y)^2}{4t}}.
\end{align*}
Combining this bound with \eqref{funda} we conclude that for $X_1\in \mathbb B_1$ and $X\in A$ 
\begin{equation}\label{g1}
\G(X_1, X, t) \leq  C\ t^{-\frac{n+1}{2}} e^{-\frac{|X-X_1|^2}{4t}}.
\end{equation}
If instead $a>0$, then we have $\left(\frac{y y_1}{2t}\right)^{-a/2} \leq	c(a)^{-a/2}$ for $X\in A$. Using this estimate and \eqref{gud} in \eqref{fs} we find this time
\begin{align*}
p^{(a)}(y_1,y,t) & \le 2^{-a/2}t^{a/2} (2t)^{-\frac{a+1}{2}}\left(\frac{y_1 y}{2t}\right)^{\frac{1}{2}}I_{\frac{a-1}{2}}\left(\frac{y_1 y}{2t}\right)e^{-\frac{y_1^2+y^2}{4t}}
\\
& \le C^{\star\star}(a) t^{-\frac{a+1}2} e^{-\frac{(y_1-y)^2}{4t}}.
\end{align*}
Combining this bound with \eqref{funda} we infer that for $X_1\in \mathbb B_1$ and $X\in A$ 
\begin{equation}\label{g2}
\G(X_1, X, t) \leq  C\ t^{-\frac{n+1}{2}} e^{-\frac{|X-X_1|^2}{4t}}.
\end{equation}
From \eqref{g1} and \eqref{g2} and \eqref{X} we conclude that when $X_1 \in \mathbb B_1$, $X \in A\cap(\mathbb B_2 \setminus  \mathbb B_{3/2})$ and $0<t\le 1$, the following bound holds for some universal $C>0$ and for $\ell= \max\{\frac{n+1}{2},\frac{n+1+a}{2}\}$
\[
\G(X_1, X, t) \leq  C\ t^{-\ell} e^{-\frac{9}{4t}}.
\] 
From the latter inequality \eqref{g4} immediately follows when $X\in A\cap(\mathbb B_2 \setminus  \mathbb B_{3/2})$. If instead $X\in A^c\cap(\mathbb B_2 \setminus  \mathbb B_{3/2})$, keeping in mind that on the set  $A^c$ we have $\frac{yy_1}{2t} \le c(a)$, by the first inequality in \eqref{bessel} we obtain that for all $a \in (-1, 1)$,
\[
I_{\frac{a-1}{2}}\left(\frac{yy_1}{2t}\right) \leq C(a) \left(\frac{yy_1}{2t}\right)^{\frac{a-1}{2}}. 
\]
Using this in \eqref{fs} we find
\[
p^{(a)}(y_1,y,t) \le C(a) (2t)^{-\frac{a+1}{2}} e^{-\frac{y_1^2+y^2}{4t}} \le C^\star(a) t^{-\frac{a+1}{2}} e^{-\frac{(y_1-y)^2}{4t}} .
\]
Combining this bound with \eqref{funda} we conclude that for $X_1\in \mathbb B_1$, $0<t\le 1$ and $X\in A^c$ 
\begin{equation*}
\G(X_1,X,t) \leq Ct^{-\frac{n+1+a}{2}} e^{-\frac{|X-X_1|^2}{4t}}.
\end{equation*}
Combining this estimate with \eqref{X} we conclude that  \eqref{g4} also holds for any $X_1 \in \mathbb B_1$, $0<t\le 1$ and $X \in A^c \cap (\mathbb B_2 \setminus  \mathbb B_{3/2})$.

Having proved \eqref{g4} we now insert such inequality in the definition of $I_1$ and using \eqref{feq} and \eqref{obs1} we finally obtain 
\[
|I_1| \le C e^{-\frac{1}{Mt}} \int_{\mathbb B_2} \left(|\nabla U| + |U|\right) |U| y^a.
\]
At this point we invoke the $L^{\infty}$ bounds for $U, \nabla_x U, U_t$ and $y^a U_y$ in Lemma \ref{reg1} to finally conclude that for every $X_1\in \mathbb B_1$ and $0<t\le 1$ the inequality \eqref{i1} holds.

Finally, concerning \eqref{traceapp} we note that \eqref{zero} gives
\[
p^{(a)}(y_1,0,t) = \frac{2^{-a}}{\Gamma((1+a)/2)} t^{-\frac{a+1}{2}} e^{-\frac{y_1^2}{4t}}.
\] 
Combining this observation with \eqref{funda} allows us to write  
$$ I_3 = \int_{\{y=0\}} 2Vf^2\G dx = \int_{\{y=0\}} 2 V f^2 G(x_1,x, t) dx,$$
where for a fixed $X_1\in \mathbb B_1$ we have let
\begin{equation}\label{gig}
G(x_1,x, t) = \frac{(4\pi)^{-\frac n2} 2^{-a}}{\Gamma((1+a)/2)} t^{-\frac{n+a+1}{2}} e^{-\frac{|x-x_1|^2+y_1^2}{4t}}.
\end{equation}
If we now apply the elliptic trace inequality in Lemma \ref{tr} to the function $u= f \sqrt{G} e^{-y^2/8t}$ we obtain for all $\sigma >1$
\begin{align*}
& |I_3| \leq C(n,a)||V||_{\infty}\left(\sigma ^{1+a} \int f^2 G e^{-\frac{y^2}{4t}} y^a dX + \sigma^{a-1} \int |\nabla (f \sqrt{G} e^{-\frac{y^2}{8t}})|^2 y^a dX\right)
\\
&\leq C(n,a)||V||_{\infty}\bigg(\sigma ^{1+a} \int f^2 G e^{-\frac{y^2}{4t}} y^a dX + \sigma^{a-1} \int |\nabla f|^2 G y^a dX
\\
& + \sigma^{a-1} \int f^2 G e^{-\frac{y^2}{4t}}\frac{|X-(x_1,0)|^2}{16t^2} y^a dX\bigg).
\\
\end{align*}
We emphasise at this moment that the power $t^{-2}$ in the last integral in the right-hand side is quite problematic. If, using the support property of $f$, we attempted to plainly bound such integral by $\frac{\sigma^{a-1}}{t^2} \int f^2 G e^{-\frac{y^2}{4t}} y^a dX$ we would end up with a term $t^{-(1+a)}\int f^2 \G y^a dX$ in \eqref{traceapp} and this would jeopardise the whole proof. The next computation deals with this delicate point. Integrating by parts in the second integral below it is easy to recognise that
\begin{align*}
& \int f^2 G e^{-\frac{y^2}{4t}}\frac{|X-(x_1,0)|^2}{16t^2} y^adX = - \frac{1}{8t} \int f^2 \langle\nabla (G e^{-\frac{y^2}{4t}}),(x-x_1,y)\rangle y^a dX\\
&= \frac{n+a+1}{8t} \int f^2 G e^{-\frac{y^2}{4t}} y^a dX + \frac{2}{8t} \int f \langle\nabla f,(x-x_1,y)\rangle G e^{-\frac{y^2}{4t}}y^a dX\notag\\
&\leq  \frac{n+a+1}{8t} \int f^2 G e^{-\frac{y^2}{4t}} y^a dX + 	\int f^2 G e^{-\frac{y^2}{4t}}\frac{|X-(x_1,0)|^2}{32t^2} dX\notag\\ 
& + \frac{1}{2} \int |\nabla f|^2 Ge^{-\frac{y^2}{4t}} y^a dX\notag,
\end{align*}
where in the  last inequality we have used the elementary fact $|AB| \le \frac{A^2+B^2}2$. 
Replacing this estimate in the above bound for $I_3$ we find 
\begin{align}\label{thedevil}
& |I_3| \leq C(n,a)||V||_{\infty}\bigg(\sigma^{1+a}\int f^2 G e^{-\frac{y^2}{4t}} y^a dX + \frac{n+a+1}{4t} \sigma^{a-1} \int f^2 G e^{-\frac{y^2}{4t}} y^a dX 
\\
& +  2\sigma^{a-1} \int |\nabla f|^2 Ge^{-\frac{y^2}{4t}} y^a dX\bigg).
\notag
\end{align}
If for $\nu>-1$ we now define
\[
\Lambda_\nu(z) = z^{-\nu} I_\nu(z),\ \ \ \ \ \ \ y_\nu(z) = \frac{I_{\nu+1}(z)}{I_\nu(z)},
\]
then using \cite[formula (8.15)]{Gcm}
\[
\frac{d}{dz} \log \Lambda_\nu(z) = y_\nu(z), \ \ \ \ \ \  \ \ \ \ \ \ \ z>0,
\]
we recognise that $z\to \Lambda_\nu(z)$ is strictly increasing on $(0,\infty)$. This information and the fact that
\begin{equation}\label{lambda}
\underset{z\to 0^+}{\lim} \Lambda_\nu(z) = \frac{2^{-\nu}}{\Gamma(\nu+1)},
\end{equation}
which follows from \eqref{besseries}, allow to infer that for every $z\in (0,\infty)$
\[
\Lambda_\nu(z) \ge \frac{2^{-\nu}}{\Gamma(\nu+1)}.
\]
Since in view of 
 \eqref{fs} we have
\begin{equation}\label{pa}
p^{(a)}(y_1,y,t)  =(2t)^{-\frac{a+1}{2}}\Lambda_{\frac{a-1}{2}}\left(\frac{y_1 y}{2t}\right) e^{-\frac{y_1^2+y^2}{4t}},
\end{equation}
we conclude from \eqref{funda}, \eqref{pa} and the definition \eqref{gig} of $G(x_1,x, t)$ that
\begin{equation}\label{belowbd}
\G(X_1,X,t) \geq  G(x_1,x,t)  e^{-\frac{y^2}{4t}}.
\end{equation}
Using \eqref{belowbd} in \eqref{thedevil} we obtain the following crucial bound
\begin{align}\label{mk1}
|I_3| & \leq C(n,a)||V||_{\infty}\bigg(\sigma^{1+a}\int f^2 \G y^a dX + \frac{n+a+1}{4t} \sigma^{a-1} \int f^2 \G y^a dX 
\\
& + \sigma^{a-1} \int |\nabla f|^2 \G y^a dX\bigg).
\notag
\end{align}
Letting now $\sigma = 1/\sqrt{t}$ we finally obtain \eqref{traceapp} from \eqref{mk1}, thus completing the proof of the lemma. 
\end{proof}
With Lemma \ref{mont} in hand, our next objective is establishing a a conditional elliptic type doubling inequality which is the key tool in the proof of Theorem \ref{main}. Before we can do that, however, we need to establish several crucial auxiliary tools and we thus turn to setting the stage for them. In what follows, with $b>0$ to be chosen appropriately in the proof of Theorem \ref{db1} below, with a slight abuse of notation, for every $X = (x,y)\in \R^{n+1}_+$ and $t>-b$  we will indicate with $\G(X,t+b)$ the function $\G(X,0,t+b)$ with pole at $(0,-b)\in \R^{n+1}\times \R$, and whenever convenient we will simply write $\G$ instead of $\G(X,t+b)$. We explicitly note the following important consequence of \eqref{funda}, \eqref{pa} 
\begin{equation}\label{Gpole}
\G(X,t+b) = \frac{(4\pi)^{-\frac n2} 2^{-a}}{\Gamma((1+a)/2)} (t+b)^{-\frac{n+a+1}{2}} e^{-\frac{|x|^2 + y^2}{4(t+b)}}.
\end{equation}
Formula \eqref{Gpole} gives in particular
\begin{equation}\label{nG}
\n \G(X,t+b) = - \frac{X}{2(t+b)} \G(X,t+b).
\end{equation}
For future reference we note that, in view of \eqref{Gpole},  we have the following obvious estimate for an appropriate $C(n,a)>0$
\begin{equation}\label{Gring}
\ \ \ \ \ \ \ \ \ \ \ \ \ \ \ \ \ \G(X,t) \le \frac{C(n,a)}{(t+b)^{\frac{n+a+3}{2}}} e^{-\frac{9}{4 (t+b)}},\ \ \ \ \ \ X\in \mathbb B_{7/2} \setminus \mathbb B_3,\ \ t+b>0.
\end{equation}
Also, if we let
$$G_b(x,t) = \frac{(4\pi)^{-\frac n2} 2^{-a}}{\Gamma((1+a)/2)} (t+b)^{-\frac{n}{2}}e^{-\frac{|x|^2}{4(t+b)}},$$
then it is clear from \eqref{funda} and \eqref{Gpole} that
\begin{equation}\label{screwy}
\G((x,0),t+b) =  (t+b)^{-\frac{a+1}{2}} G_b(x,t).
\end{equation}
Using \eqref{screwy} and the equation $\p_t G_b = \Delta G_b$ satisfied by $G_b$, we see that on the thin set $\{y=0\}$ we have
\begin{equation}\label{screwy2}
\G_t = - \frac{a+1}2 (t+b)^{-1} \G + (t+b)^{-\frac{a+1}2} \Delta G_b,\ \ \ \ \ \ \n_x \G = - \frac{x}{2(t+b)} \G.
\end{equation}
In all computations in the remainder of this section we stick to the agreement adopted in the proof of Lemma \ref{mont} that we do not indicate the domain of integration if an integral is taken on $\R^{n+1}_+$. In some computations it will be easier to have a unified notation for the coordinates of $X$, and so we will let  $x_{n+1} = y$. Consequently, derivatives with respect to the variable $y$ will be indicated with $D_{n+1}$. Similarly, $D_{n+1,n+1}, D_{i,n+1}$, $i=1,...,n$ stand for $\p_{yy}, D_{x_i y}$, etc. Whenever convenient the summation convention over repeated indices will be used, so for instance $|\nabla f|^2 = f_i f_i$, where $i = 1,...,n+1$. 

We now consider a solution $U$ of \eqref{exprob} in $\mathbb Q_4$. We fix a spherically symmetric cutoff $\phi \in C_0^{\infty}(\mathbb B_{7/2})$ such that $0\le \phi\le 1$ and $\phi \equiv1$ on $\mathbb B_{3}$, and let  $f = \phi U$. As in the proof of Lemma \ref{mont} $f$ solves the problem \eqref{feq}. We also have that $\frac{\phi_y}{y}$ is bounded up to $\{y=0\}$ and 
\begin{equation}\label{obs01}
\begin{cases}
\operatorname{supp} (\nabla \phi) \cap \{y>0\}  \subset \mathbb B_{7/2} \setminus \mathbb B_{3},
\\
|\D(y^a \nabla \phi)| \leq C y^a\ \mathbf 1_{\mathbb B_{7/2} \setminus \mathbb B_{3}}.
\end{cases}
\end{equation}
Given such $f$ we next introduce the following shifted (in time) \emph{height} functional
\begin{equation}\label{Hb}
H_b(t) = \int_{\R^{n+1}_+} f(X,t)^2 \G(X,t+b) y^a dX,
\end{equation} 
and \emph{energy}
\begin{equation}\label{Db}
D_b(t)=\int_{\R^{n+1}_+} |\nabla f(X,t)|^2 \G(X,t+b)y^adX +  \int_{\{y=0\}}V(x,t)f((x,0),t)^2\G((x,0),t+b)dx.
\end{equation}
We also consider the following shifted \emph{frequency} functional 
\begin{equation}\label{Nb}
N_b(t)=\frac{I_b(t)}{H_b(t)},
\end{equation}
where we have let 
\begin{equation}\label{Ib}
I_b(t)=(t+b) D_b(t).
\end{equation}
The reader should keep in mind that $N_b(t)$ is well defined for all $t>0$ sufficiently small since the hypothesis \eqref{ass} implies that $H_b(t)>0$ for all such $t$.  
We emphasise that, despite the similarity between \eqref{Hb}, \eqref{Db} and \eqref{Nb} and the quantities first introduced in the previous work \cite{BG}, the present situation is completely different since the presence of the cutoff function $\phi$ in their definitions complicates matters considerably. We also mention that, in a completely different context, related nonlocalised versions of the quantities $H_b, D_b$ and $N_b$ have been used in the work \cite{BDGP}.

In the proof of Theorem \ref{db1} it will be important to have the following bound from below for $H_b(t)$.
\begin{lemma}\label{L:below}
There exists $N = N(n,a,||V||_1)>2$ such that for $0< t \le t+b \leq 1/{N\log(N\theta)}$ one has 
\begin{equation}\label{kju}
H_b(t) \geq  N^{-1} e^{-\frac{1}{t+b}} \int_{\mathbb B_1} U(X,0)^2 y^a dX.
\end{equation}
\end{lemma}

\begin{proof} We begin by observing that in view of \eqref{Gpole} we have for every $X\in\mathbb B_2$
\[
\G(X,t+b) \ge \frac{(4\pi)^{-\frac n2} 2^{-a}}{\Gamma((1+a)/2)} (t+b)^{-\frac{n+a+1}{2}} e^{-\frac{1}{t+b}} \ge e^{-\frac{1}{t+b}},
\]
provided $t+b\le C(n,a)$, for an appropriate $C(n,a)>0$. Keeping in mind that $\phi \equiv 1$ on $\mathbb B_2$, from this observation and \eqref{Hb} we find
\[
H_b(t) \ge \int_{\mathbb B_2} U(X,t)^2 \G(X,t+b) y^a dX \ge e^{-\frac{1}{t+b}} \int_{\mathbb B_2} U(X,t)^2  y^a dX.
\]
Applying Lemma \ref{mont} we now infer the existence of a constant $N = N(n,a,||V||_1)>2$ such that $N\operatorname{log}(N\theta) \geq 1$ and for which the following inequality holds for $0\leq t \leq 1/{N\operatorname{log}(N\theta)}$
\begin{align*}
N\int_{\mathbb B_2} U(X,t)^2 y^adX \geq \int_{\mathbb B_1} U(X,0)^2 y^adX.
\end{align*}	
It is clear that, by possibly adjusting the choice of $N$ so that $N\log(N\theta)\ge \max\{1,C(n,a)^{-1}\}$, we obtain the desired conclusion \eqref{kju}.

\end{proof}

In the next statement we indicate with $Z$ the vector field whose action on a function $F(X,t)$ is defined by the formula
\begin{equation}\label{Z}
Z F= \langle\n F,\frac{X}{2(t+b)}\rangle + F_t.
\end{equation}

\begin{lemma}[First variation estimates]\label{L:Db}
There exist constants $C = C(n,a,||V||_1)>0$ and $0<t_0 = t_0(n,a,||V||_{1})<1$ such that for $0<t+b\le t_0$ one has:
\begin{align}\label{Hprime}
\left|H_b'(t) -2 \int \phi^2 U ZU \G y^a dX\right|	 \leq 	  e^{-\frac{17}{8 (t+b)}} \int_{\mathbb Q_4} U^2 y^a dXdt,	
\end{align}
\begin{equation}\label{hb8}
\bigg| \frac{H_b'(t)}{H_b(t)}  -  \frac{2}{t+b} N_b(t) \bigg|\leq   \frac{e^{-\frac{17}{8(t+b)}}\int_{\mathbb Q_4} U^2 y^a dXdt}{H_b(t)},
\end{equation}
and also
\begin{align}\label{db2}
D_b'(t) & \geq 2\int \phi^2 (Z U)^2 \G y^a dX -  e^{-\frac{17}{8 (t+b)}} \int_{\mathbb Q_4} U^2 y^a dXdt
\\
& - \left(\frac{1}{t+b} + \frac{C}{(t+b)^{\frac{a+1}{2}}}\right)D_b(t)- \frac{C}{(t+b)^{\frac{3+a}{2}}} H_b(t).
\notag
\end{align}
\end{lemma} 

\begin{proof}
Differentiating \eqref{Hb} and integrating by parts using \eqref{neu} and \eqref{screwy2}, we find
\begin{align*}
H_b'(t)&=\int 2 f f_t \G y^a dX + \int  f^2 \G_t y^a dX = \int 2 f f_t \G y^a dX + \int  f^2 \D(y^a\n \G) dX
\\
& = \int 2 f f_t \G y^a dX  -\int \langle\nabla f^2, \nabla \G\rangle y^a dX
\\
&= 2\int f\big\{f_t + \langle\n f,\frac{X}{2(t+b)}\rangle \big\} \G y^a dX = 2 \int f\ Zf\ \G y^a dX,
\end{align*}
where in the second to the last equality we have used \eqref{nG}, while in the last we have used \eqref{Z}. 
Keeping in mind that $f=U\phi$, by some elementary computations we obtain
\[
H'_b(t) = 2 \int \phi^2 U ZU  \G y^a dX + 2 \int \phi Z\phi  U^2  \G y^a dX.
\]
Since $Z\phi = \langle\n \phi,\frac{X}{2(t+b)}\rangle$ is supported in $\mathbb B_{7/2} \setminus \mathbb B_3$, by the latter equality and \eqref{Gring} we obtain for every $t+b>0$
\begin{align}\label{pert}
& \left|H_b'(t) -2 \int \phi^2 U ZU \G y^a dX\right| \leq \frac{C(n)}{t+b} \int_{\mathbb B_{7/2} \setminus \mathbb B_3} U^2 \G y^a dX
\\
& \le \frac{C(n,a)}{(t+b)^{\frac{n+a+3}{2}}} e^{-\frac{9}{4 (t+b)}} \int_{\mathbb B_4} U^2  y^a dX.
\notag
\end{align}
If we now choose $t_0<1$ sufficiently small, depending on the constant $C(n,a)$, it is clear that for $(t+b) \leq t_0$ we can ensure that
\begin{equation}\label{betteree}
\frac{C(n,a)}{(t+b)^{\frac{n+a+3}{2}}} e^{-\frac{9}{4 (t+b)}}  \leq e^{-\frac{17}{8 (t+b)}}, 
\end{equation}
thus establishing \eqref{Hprime}. To prove \eqref{hb8} observe that repeating with $\G(X,t+b)$ the computations in \eqref{h1} we find
\begin{align*}
H_b'(t)&=2\int f(f_t + \D(y^a \n f))\G y^a dX + 2D_b(t).
\end{align*}
Using \eqref{feq} and again the fact that  $\nabla \phi$ and $\D(y^a \nabla \phi)$ are supported in $\mathbb{B}_{7/2}  \setminus \mathbb B_3$, arguing as for \eqref{pert}, \eqref{betteree} we obtain \eqref{hb8}.

Next, we prove \eqref{db2}. Differentiating \eqref{Db} and using the Neumann condition $\py f((x,t),0)= V(x,t) f(x,t)$ in \eqref{feq} we obtain 
\begin{align*}
&D_b'(t)= 2 \int  \langle \n f,\n f_{t}\rangle\G y^a dX + \int |\nabla f|^2 \G_t y^a dX \\
&+ \int_{\{y=0\}} V_t f^2 \G dx +\int_{\{y=0\}}2Vff_t\G dx +\int_{\{y=0\}}Vf^2 \G_t dx.
\end{align*}
Integrating by parts and using \eqref{feq} again we find
\begin{align*}
& 2 \int  \langle \n f,\n f_{t}\rangle\G y^a dX  =  -2\int \operatorname{div}(y^a \G\n f) f_t dX -\int_{\{y=0\}}2Vff_t\G dx\\
	&= -2 \int f_t \D(y^a \nabla f)\G  dX -2\int f_t \langle \nabla f,\nabla \G\rangle y^a dX -\int_{\{y=0\}}2Vff_t\G dx\notag.
\end{align*}
Substituting in the above expression of $D_b'(t)$ we obtain
\begin{align}\label{0th}
&D_b'(t)= -2 \int f_t \D(y^a \nabla f)\G  dX -2\int f_t \langle \nabla f,\nabla \G\rangle y^a dX  + \int |\nabla f|^2 \G_t y^a dX \\
&+ \int_{\{y=0\}} V_t f^2 \G dx +\int_{\{y=0\}}Vf^2 \G_t dx.
\notag
\end{align}
Next, keeping in mind that $y^a \partial_t \G - \operatorname{div} (y^a \nabla \G) = 0$, an integration by parts combined with \eqref{neu}, \eqref{feq} and \eqref{nG} give
\begin{align}\label{nabla2}
& \int |\nabla f|^2 \G_t y^a dX = \int |\nabla f|^2 \D(y^a\n \G ) dX 
\\
& =-\int \langle\n(|\n f|^2),\n \G\rangle y^a dX = 2 \int f_i f_{ij}\frac{x_j}{2(t+b)} \G y^a dX,
\notag\\
&=-2\int D_i \left(y^a f_i \frac{x_j}{2(t+b)} \G\right) f_j dX +2 \its V f \langle  \n_x \G , \n_x f  \rangle dx
\notag\\
&=-2 \int \D(y^a \n f)\langle\frac{X}{2(t+b)},\n f \rangle\G dX - \frac{2}{2(t+b)} \int  |\n f|^2 \G y^a dX
\notag\\ 
&+ 2 \int \langle\frac{X}{2(t+b)},\n f\rangle^2 \G y^a  dX + 2 \its V f \langle \n_x \G, \n_x f \rangle dx.
\notag
\end{align}
In the ensuing computations we simplify the term $\its f^2 V \G_t dx$ in the right-hand side of \eqref{0th}. Using \eqref{screwy2} and the fact that the integrands in the integrals on the thin set $\{y=0\}$ are compactly supported in $B_{7/2}$, we find 
\begin{align}\label{2nd}
& \its f^2 V \G_t dx = -\frac{a+1}{2}\its f^2 V (t+b)^{-\frac{a+1}{2}-1} G_b dx + (t+b)^{-\frac{a+1}{2}}\its V f^2 \Delta_x G_b dx\\
&= -\frac{a+1}{2}\its f^2 V (t+b)^{-\frac{a+1}{2}-1} G_b dx - (t+b)^{-\frac{a+1}{2}}\its \langle \n_x (V f^2), \n_x G_b \rangle dx
\notag\\
&=-\frac{a+1}{2}\its f^2 V (t+b)^{-\frac{a+1}{2}-1} G_b dx - (t+b)^{-\frac{a+1}{2}}\its f^2 \langle \n_x V, \n_x G_b \rangle dx
\notag\\
& - 2(t+b)^{-\frac{a+1}{2}} \its  f V\langle \n_x f ,\n_x  G_b\rangle dx.\notag
\end{align}
Inserting \eqref{nabla2} and \eqref{2nd} in \eqref{0th}, and noting that in view of \eqref{screwy} we have
\[
2 \its V f \langle \n_x \G, \n_x f \rangle dx- 2(t+b)^{-\frac{a+1}{2}} \its  f V\langle \n_x f ,\n_x  G_b\rangle dx = 0,
\]
we obtain
\begin{align}\label{db}
& D_b'(t) = -2 \int f_t \D(y^a \nabla f)\G  dX -2\int f_t  \langle \nabla f, \nabla \G \rangle y^a dX
\\ 
& -2 \int \D(y^a \n f) \langle \frac{X}{2(t+b)},\n f\rangle \G dX - \frac{2}{2(t+b)} \int y^a |\n f|^2 \G dX
\notag\\
& + 2 \int y^a   \langle \n f , \frac{X}{2(t+b)} \rangle^2 \G dX
+ \its V_t f^2 \G dx 
\notag\\
& -\frac{a+1}{2}\its f^2 V (t+b)^{-\frac{a+1}{2}-1} G_b dx - \frac{1}{(t+b)^{\frac{a+1}{2}}}\its f^2 \langle \n_x V, \n_x G_b \rangle dx\notag.
\end{align}
Next we observe that in view of \eqref{nG} the following non-boundary terms in the right-hand side of \eqref{db} can be expressed as follows
\begin{align}\label{ter1}
&-2 \int f_t \D(y^a \nabla f)\G  dX -2\int f_t  \langle \nabla f, \nabla \G \rangle y^a dX -2 \int \D(y^a \n f) \langle \n f,\frac{X}{2(t+b)}\rangle \G dX\\
&  + 2 \int y^a   \langle\n f ,\frac{X}{2(t+b)}\rangle^2\G dX
 = 2 \int	\bigg\{- f_t \D(y^a \n f)  +  f_t \langle \n f, \frac{X}{2(t+b)}\rangle  y^a
\notag\\
& - \D(y^a \n f)\langle\n f ,\frac{X}{2(t+b)} \rangle +  \langle\n f , \frac{X}{2(t+b)}\rangle^2 y^a \bigg\} \G dX.
\notag
\end{align}
Keeping in mind that $f=U\phi$, and that $y^a U_t = - \D(y^a \n U)$, see \eqref{exprob}, we observe that  after some elementary calculations the expression within curly brackets in the integral in the right-hand side of \eqref{ter1} can be rewritten in the following way
\begin{align}\label{tyer1}
&- f_t \D(y^a \n f)  +  f_t \langle \n f, \frac{X}{2(t+b)}\rangle  y^a
 - \D(y^a \n f)\langle\n f ,\frac{X}{2(t+b)} \rangle +  \langle\n f , \frac{X}{2(t+b)}\rangle^2 y^a 
\\
& = \phi^2 \left\{U^2_t + 2 U_t \langle\n U,\frac{X}{2(t+b)}\rangle + \langle\n U,\frac{X}{2(t+b)}\rangle^2\right\} y^a
\notag\\
&  - 2 \phi U_t \langle \n U,\n\phi\rangle y^a - \phi U U_t \D(y^2 \n \phi) + 2 \phi U U_t \langle\n \phi,\frac{X}{2(t+b)}\rangle y^a
\notag\\
& - 2 \phi \langle \n U,\n\phi\rangle \langle\n U,\frac{X}{2(t+b)}\rangle - 2 U \langle \n U,\n\phi\rangle \langle\n \phi,\frac{X}{2(t+b)}\rangle y^a
\notag\\
& - \phi U \langle \n U,\frac{X}{2(t+b)}\rangle \D(y^a \n \phi) - U^2 \langle \n \phi,\frac{X}{2(t+b)}\rangle \D(y^a \n \phi)
\notag\\
& + U^2 \langle\n \phi,\frac{X}{2(t+b)}\rangle^2 y^a + 2 \phi U \langle\n U,\frac{X}{2(t+b)}\rangle\langle\n \phi,\frac{X}{2(t+b)}\rangle y^a.
\notag
\end{align}
Keeping \eqref{Z} in mind it is clear from \eqref{tyer1} that 
\begin{align}\label{tyer2}
&- f_t \D(y^a \n f)  +  f_t \langle \n f, \frac{X}{2(t+b)}\rangle  y^a
 - \D(y^a \n f)\langle\n f ,\frac{X}{2(t+b)} \rangle +  \langle\n f , \frac{X}{2(t+b)}\rangle^2 y^a 
\\
& = \phi^2 (ZU)^2 y^a + \mathscr R(U,\phi)(X,t),
\notag
\end{align}
where we have collected under the symbol $\mathscr R(U,\phi)(X,t)$ the remaining terms in \eqref{tyer1} which have $\nabla \phi$ and  $\D(y^a \n\phi)$ in their expression. Since the latter two functions are supported in $\mathbb B_{7/2} \setminus \mathbb B_{3}$, in view of \eqref{obs01}  
and of the regularity estimates in Lemma \ref{reg1}, it is not difficult to verify that  
\[
|\int_{\R^{n+1}_+} \mathscr R(U,\phi)(X,t) dX| \le \frac{C(n,a, V)}{(t+b)^{\frac{n+a+3}{2}}} e^{-\frac{9}{4 (t+b)}}   \int_{\mathbb Q_4} U^2 y^a dXdt.
\]
If we now choose $t_0<1$ sufficiently small, depending on the constant $C(n,a,V)$, it is clear that for $(t+b) \leq t_0$ we can ensure that \eqref{betteree} holds, with $C(n,a,V)$ instead of $C(n,a)$. We thus obtain
\begin{equation}\label{hb0}
|\int_{\R^{n+1}_+} \mathscr R(U,\phi)(X,t) dX| \le e^{-\frac{17}{8 (t+b)}} \int_{\mathbb Q_4} U^2 y^a dXdt.
\end{equation}
Combining \eqref{ter1}-\eqref{hb0} with \eqref{db}, and keeping \eqref{screwy} and \eqref{Db} in mind, we obtain
\begin{align}\label{doobydoo}
D_b'(t) & \ge 2\int \phi^2 (Z U)^2 \G y^a dX - e^{-\frac{17}{8 (t+b)}} \int_{\mathbb Q_4} U^2 y^a dXdt - \frac{1}{t+b} D_b(t)  
\\
& \ + \frac{1}{t+b} \int_{\{y=0\}}Vf^2 \G dx + \its V_t f^2 \G dx 
\notag\\
& \ -\frac{a+1}{2(t+b)}\its V f^2  \G dx  - \frac{1}{(t+b)^{\frac{a+1}{2}}}\its f^2 \langle \n_x V, \n_x G_b \rangle dx.
\notag\\
& = 2\int \phi^2 (Z U)^2 \G y^a dX - e^{-\frac{17}{8 (t+b)}} \int_{\mathbb Q_4} U^2 y^a dXdt - \frac{1}{t+b} D_b(t)  
\notag\\
& \ + \frac{1-a}{2(t+b)} \int_{\{y=0\}}Vf^2 \G dx + \its V_t f^2 \G dx 
\notag\\
&   - \frac{1}{(t+b)^{\frac{a+1}{2}}}\its f^2 \langle \n_x V, \n_x G_b \rangle dx.
\notag
\end{align}
Finally, we need to appropriately bound the three integrals on the thin set $\{y=0\}$ in the right-hand side of \eqref{doobydoo}. With this in mind, using \eqref{screwy2} we obviously have
\begin{align}\label{herewego}
& \left|\frac{1-a}{2(t+b)} \int_{\{y=0\}}Vf^2 \G dx + \its V_t f^2 \G dx 
- \frac{1}{(t+b)^{\frac{a+1}{2}}}\its f^2 \langle \n_x V, \n_x G_b \rangle dx\right| 
\\
& \le C(n,a,||V||_1)\ \frac{1}{t+b} \its f^2 \G dx.
\notag
\end{align} 
Applying the elliptic trace inequality in Lemma \ref{tr} to the function $f\sqrt \G = f(X,t)\sqrt{\G(X,t+b)}$ with a constant $\sigma = (t+b)^{-\frac{1}{2}}$, and using \eqref{Db}, we find
\begin{align}\label{tracef}
& \its f^2 \G dx \leq C_0 (t+b)^{-\frac{1+a}{2}}\int f^2 \G y^a dX + C_0 (t+b)^\frac{1-a}{2} \int |\n f |^2 \G y^a dX\\
&=C_0 \bigg((t+b)^{-\frac{1+a}{2}} H_b(t) + (t+b)^\frac{1-a}{2} D_b(t) - (t+b)^\frac{1-a}{2} \its V f^2 \G dx\bigg)\notag\\
&\leq C_0 \bigg((t+b)^{-\frac{1+a}{2}}H_b(t) + (t+b)^\frac{1-a}{2} D_b(t) + ||V||_{\infty}(t+b)^\frac{1-a}{2} \its f^2 \G dx\bigg)\notag.
\end{align}
If we now choose $t_0<1$ such that $C_0 ||V||_{\infty}t_0^{\frac{1-a}{2}} <1/2$, then the last integral in the right-hand side of \eqref{tracef} can be absorbed in the left-hand side and for $t+b \leq t_0$ we finally have 
\begin{align}\label{j1}
	\its f^2 \G dx \leq C_0 \bigg((t+b)^{-\frac{1+a}{2}}H_b(t) + (t+b)^{\frac{1-a}{2}} D_b(t)\bigg).
\end{align}
Inserting \eqref{j1} into \eqref{herewego}, and then \eqref{herewego} into \eqref{doobydoo}, we finally reach the desired conclusion \eqref{db2}.

\end{proof}

We now record as a corollary an elementary observation which will be needed in the proof of Theorem \ref{db1} below.

\begin{cor}
With the constants $C=C(n, a, ||V||_1)>0$ and $0<t_0 = t_0(n,a,||V||_{1})<1$ as in Lemma \ref{L:Db} we have for every $0<t+b\le t_0$
\begin{equation}\label{bnb}
N_b(t) \geq - C (t+b)^{\frac{1-a}{2}}.
\end{equation}
Furthermore, we have
\begin{equation}\label{bnb1}
|N_b(t)| \leq  2 C (t+b)^{\frac{1-a}{2}} + N_b(t).
\end{equation}
\end{cor}

\begin{proof}
From \eqref{Db} we trivially have 
\[
D_b(t) \geq \int_{y=0} Vf^2 \G dx.
\]
Applying \eqref{j1} above we obtain
\begin{align*}
& D_b(t) \geq - C (t+b)^{\frac{1-a}{2}} D_b(t) - C (t+b)^{-\frac{1+a}{2}} H_b(t),
\end{align*}
which we can rewrite
\begin{align*}
&D_b(t) \geq - \frac{C (t+b)^{-\frac{1+a}{2}} H_b(t)}{ 1+ C(t+b)^{\frac{1-a}{2}}} \geq -C (t+b)^{-\frac{1+a}{2}} H_b(t).\end{align*}
The inequality \eqref{bnb} follows from \eqref{Nb} and \eqref{Ib}, after we divide by $H_b(t)$. The inequality \eqref{bnb1} is a direct consequence of \eqref{bnb}. Note that when $N_b(t) \geq 0$ the inequality is trivially true.
If instead $N_b(t) \leq 0$, then from \eqref{bnb} we have  
\begin{equation*}
2 C (t+b)^{\frac{1-a}{2}} + N_b(t) \geq C(t+b)^{\frac{1-a}{2}} \geq |N_b(t)|.
\end{equation*}

\end{proof}

We are now ready to establish the conditional doubling property which is the key  ingredient in our blowup analysis in the proof of Theorem \ref{main}. Throughout the remainder of this section the number $\theta$ has the meaning specified in \eqref{theta}. 

\begin{thrm}\label{db1}
Let $U$ be a solution of \eqref{exprob} in $\mathbb Q_4.$ There exists $N>2$, depending on $n$, $a$ and the $C^1$-norm of $V$, for which $N\log(N\theta) \ge 1$ and such that:
\begin{itemize}
\item[(i)] For $r \leq 1/2,$ we have 
$$\int_{\mathbb B_{2r}}U(X,0)^2 y^adX \leq (N \theta)^N\int_{\mathbb B_{r}}U(X,0)^2 y^adX.$$
Moreover for  $r \leq 1/\sqrt{N \operatorname{log}(N \theta)}$  the following two inequalities hold:
\item[(ii)] $$\int_{\mathbb Q_{2r}} U(X,t)^2y^a dXdt \leq \operatorname{exp}(N \operatorname{log}(N \theta) \operatorname{log}(N \operatorname{log}(N \theta)))r^2 \int_{\mathbb B_r}U^2(X,0)y^adX.$$
	 	\item[(iii)]$$\int_{\mathbb Q_{2r}} U(X,t)^2 y^adXdt \leq \operatorname{exp}(N \operatorname{log}(N \theta) \operatorname{log}(N \operatorname{log}(N \theta)))\int_{\mathbb Q_r}U(X,t)^2y^adXdt.$$ 
\end{itemize}
\end{thrm}

\begin{proof}
We begin by observing that by the divergence theorem and \eqref{feq} we can rewrite \eqref{Db} in the  following way
\begin{align*}
D_b(t)&= -\int \D(y^a \n f )f \G dX - \int \langle \n f, \n \G \rangle f y^a dX
\\
& =  \int f \big(f_t + \langle \n f,\frac{X}{2(t+b)}\rangle\big) \G  y^a dX -  \int f \left(2 y^a \langle\nabla U,\nabla \phi\rangle  + \D(y^a \nabla \phi) U\right) \G dX
\\
& = \int f Zf \G y^a dX - \int f \big(2 y^a \langle\nabla U,\nabla \phi\rangle  + \D(y^a \nabla \phi) U\big) \G dX,
\end{align*}
where in the last equality we have used \eqref{Z}. Since $f=U\phi$,  we find from this
\begin{align*}
D_b(t)&= \int \phi^2 U ZU \G y^a dX + \int \phi Z\phi U^2  \G y^a dX
\\
& - \int \phi U \left(2 y^a \langle\nabla U,\nabla \phi\rangle  + \D(y^a \nabla \phi) U\right) \G dX.
\end{align*}
By the support property of $\phi$, arguing as in \eqref{pert}, \eqref{betteree}, we infer the existence of $t_0<1$ such that for every $t+b\le t_0$
\begin{equation}\label{dasym}
\left|D_b(t)- \int \phi^2 U ZU  \G y^a dX\right|  \leq  e^{-\frac{17}{8 (t+b)}} \int_{\mathbb Q_4} U^2 y^a dXdt. 
\end{equation}
We next observe that
\begin{align}\label{clm1}
& \bigg| N_b(t)\frac{H_b'(t)}{H_b(t)} - \frac{2(t+b)}{H_b^2(t)}\left(\int U ZU \phi^2 \G y^a dX\right)^2 \bigg|
\\
& \leq  C \frac{e^{-\frac{17}{8 (t+b)}} \int_{\mathbb Q_4} U^2 y^a dXdt}{H_b(t)} |N_b(t)|    + C\bigg(\frac{e^{-\frac{17}{8 (t+b)}} \int_{\mathbb Q_4} U^2 y^a dXdt}{H_b(t)}\bigg)^2.\notag
\end{align}
To see this note that \eqref{dasym} gives
\begin{equation}\label{clm2}
D_b(t) = \int \phi^2 U ZU  \G y^a dX + \mathscr R_1,
\end{equation}
where 
\begin{equation*}
|\mathscr R_1| \leq e^{-\frac{17}{8(t+b)}} \int_{\mathbb Q_4} U^2 y^a dXdt.
\end{equation*}
Similarly, \eqref{Hprime} gives
\begin{equation}\label{clm5}
H_b'(t) = 2\int \phi^2 U ZU  \G y^a dX dt + \mathscr R_2,
\end{equation}
where
\begin{equation*}
|\mathscr R_2| \leq e^{-\frac{17}{8(t+b)}} \int_{\mathbb Q_4} U^2 y^a dXdt. \end{equation*}
Combining \eqref{clm2} and \eqref{clm5} we find
\begin{align}\label{clm8}
& N_b(t) \frac{H_b'(t)}{H_b(t)}= 2(t+b) \bigg( \frac{\int \phi^2 U ZU y^a dX}{H_b(t)} \bigg)^2
 + \frac{2 \mathscr R_1 (t+b) \int \phi^2 U ZU  \G y^a dX}{ H_b^2(t)}
 \\
 &  + \frac{\mathscr R_2 (t+b) \int \phi^2 U ZU \G y^a dX}{ H_b^2(t)} + \frac{\mathscr R_1\mathscr R_2}{H_b^2(t)}.
\notag
\end{align}
Again from \eqref{clm2} we obtain for $k = 1, 2$
\begin{align*}
&|\mathscr R_k  \int U ZU \phi^2 \G y^a dX|  \leq |D_b(t)|e^{-\frac{17}{8(t+b)}} \int_{\mathbb Q_4} U^2 y^a dXdt  + \bigg(e^{-\frac{17}{8(t+b)}} \int_{\mathbb Q_4} U^2 y^a dXdt\bigg)^2.
\end{align*}
Using this estimate in \eqref{clm8} we finally obtain \eqref{clm1}.

Our next objective is to prove the following almost monotonicity result for $N_b(t)$ in \eqref{Nb}.

\noindent \emph{\underline{Claim I}}: There exist  $C=C(n,a, ||V||_{1})>0$ and $N = N(n,a, ||V||_{1}) >2$, for which $N\log(N\theta) \ge 1$, and such that
for all $t>0$  with $t+b \leq \frac{1}{N\operatorname{log}(N\theta)}$ the following differential inequality is satisfied 
\begin{align}\label{monb}
\frac{d}{dt} \left(e^{C(t+b)^{\frac{1-a}{2}}}(N_b(t) +1)\right) \geq 0.
\end{align}
We leave it to the reader to verify that \eqref{monb} is true if we can show that for $t+b \leq \frac{1}{N\operatorname{log}(N\theta)}$ we have
\begin{equation}\label{logg}
N_b'(t) \ge -  \frac{C(1-a)}{2}(t+b)^{-\frac{1+a}{2}} (N_b(t)+1).
\end{equation}
To establish \eqref{logg} we first observe that \eqref{Nb} gives
\begin{align*}
& N_b'(t) =   \frac{t+b}{H_b(t)} D_b'(t) - N_b(t) \frac{H_b'(t)}{H_b(t)} + \frac{D_b(t)}{H_b(t)}.
\end{align*}
We now use the first variation estimate in  \eqref{db2}  to obtain from the latter equation
\begin{align}\label{log1}
& N_b'(t) \ge  \frac{t+b}{H_b(t)} \bigg\{2\int \phi^2 (Z U)^2 \G y^a dX -  e^{-\frac{17}{8 (t+b)}} \int_{\mathbb Q_4} U^2 y^a dXdt
\\
& - \left(\frac{1}{t+b} + \frac{C}{(t+b)^{\frac{a+1}{2}}}\right)D_b(t)- \frac{C}{(t+b)^{\frac{3+a}{2}}} H_b(t)\bigg\}\notag
\\
& - N_b(t) \frac{H_b'(t)}{H_b(t)} + \frac{D_b(t)}{H_b(t)}.\notag
\end{align}
Using  the estimate \eqref{clm1} in \eqref{log1} we find
\begin{align}\label{log2}
& N_b'(t) \geq \frac{2(t+b)}{H_b(t)^2}\left\{H_b(t) \int \phi^2 (Z U)^2 \G y^a dX - \left(\int \phi^2 U ZU \G y^a dX\right)^2\right\}
\\
& - C (t+b)^{-\frac{1+a}{2}} (N_b(t)+1)- (t+b) \frac{e^{-\frac{17}{8 (t+b)}} \int_{\mathbb Q_4} U^2 y^a dXdt}{H_b(t)}\notag\\
&    -    C \frac{e^{-\frac{17}{8 (t+b)}} \int_{\mathbb Q_4} U^2 y^a dXdt}{H_b(t)} |N_b(t)|  - C \left(\frac{e^{-\frac{17}{8 (t+b)}} \int_{\mathbb Q_4} U^2 y^a dXdt}{H_b(t)} \right)^2.\notag\end{align}
If we now use \eqref{bnb1} in \eqref{log2} (after relabeling by $C$ the constant $2C$) we obtain
\begin{align}\label{log20}
& N_b'(t) \geq \frac{2(t+b)}{H_b(t)^2}\left\{H_b(t) \int \phi^2 (Z U)^2 \G y^a dX - \left(\int \phi^2 U ZU \G y^a dX\right)^2\right\}
\\
& - C (t+b)^{-\frac{1+a}{2}} (N_b(t)+1) - (t+b) \frac{e^{-\frac{17}{8 (t+b)}} \int_{\mathbb Q_4} U^2 y^a dXdt}{H_b(t)}   \notag\\
&   -C \frac{e^{-\frac{17}{8 (t+b)}} \int_{\mathbb Q_4} U^2 y^a dXdt}{H_b(t)} N_b(t)     - C (t+b)^{\frac{1-a}{2}} \frac{e^{-\frac{17}{8 (t+b)}} \int_{\mathbb Q_4} U^2 y^a dXdt}{H_b(t)}\notag\\&  - C \left(\frac{e^{-\frac{17}{8 (t+b)}} \int_{\mathbb Q_4} U^2 y^a dXdt}{H_b(t)} \right)^2
\notag\\
& \ge - C (t+b)^{-\frac{1+a}{2}} (N_b(t)+1) - (t+b) \frac{e^{-\frac{17}{8 (t+b)}} \int_{\mathbb Q_4} U^2 y^a dXdt}{H_b(t)}   \notag\\
&   -C \frac{e^{-\frac{17}{8 (t+b)}} \int_{\mathbb Q_4} U^2 y^a dXdt}{H_b(t)} N_b(t)     - C (t+b)^{\frac{1-a}{2}} \frac{e^{-\frac{17}{8 (t+b)}} \int_{\mathbb Q_4} U^2 y^a dXdt}{H_b(t)}\notag\\&  - C \left(\frac{e^{-\frac{17}{8 (t+b)}} \int_{\mathbb Q_4} U^2 y^a dXdt}{H_b(t)} \right)^2,
\notag
\end{align}
where in the second inequality we have applied Cauchy-Schwarz inequality to infer that
\begin{equation*}
\frac{2(t+b)}{H_b(t)^2}\left\{H_b(t) \int \phi^2 (Z U)^2 \G y^a dX - \left(\int \phi^2 U ZU \G y^a dX\right)^2\right\} \geq 0.
\end{equation*}
Recalling the definition \eqref{theta} of $\theta$, and using \eqref{kju} in Lemma \ref{L:below} in the inequality \eqref{log20}, we obtain
\begin{align}\label{kju2}
&N_b'(t) \geq - C (t+b)^{-\frac{1+a}{2}} (N_b(t)+1) -  C (e^{-\frac{9}{8(t+b)}} \theta) N_b(t)\\
& - C(e^{-\frac{9}{8(t+b)}} \theta) - C (e^{-\frac{9}{8(t+b)}} \theta)^2
\notag\\
& \ge C (t+b)^{-\frac{1+a}{2}} (N_b(t)+1).
\notag
\end{align}
Note that in \eqref{kju2} we have used that $t+b\le 1$ (and therefore $(t+b)^{\frac{1-a}{2}} 
\le 1$), and also that since in particular $N>1$, if $t+b \leq \frac{1}{N\operatorname{log}(N \theta)}$ then we trivially have $e^{-9/8(t+b)}\theta \leq 1$. This finishes the proof of \eqref{logg}, and therefore of the \emph{\underline{Claim I}} \eqref{monb}.

Our next objective is proving that \eqref{monb} implies that on a sufficiently small time interval the frequency $N_b(t)$ is bounded from above. Precisely, we have:

\noindent \emph{\underline{Claim II}}: There is a (possibly larger) choice of the number $N = N(n,a, ||V||_{1}) >2$ in \emph{\underline{Claim I}} such that
for all $t>0$  with $t+b \leq \frac{1}{12 N\log(N\theta)}$ one has 
\begin{equation}\label{Nbound}
N_b(t) \leq N \operatorname{log}(N\theta).
\end{equation}
To establish \eqref{Nbound}, for ease of notation we henceforth let $\beta = \frac{1}{N \log(N\theta)}$. The reader should keep in mind that $0<\beta<1$. If in \eqref{hb8} we use  the lower bound \eqref{kju} and the above observed estimate $e^{-9/8(t+b)}\theta \leq 1$ we easily deduce that, with $N>2$ as in Lemma \ref{L:below}, for $0< t \le t+b \leq \beta$ the following two estimates hold 
\begin{equation}\label{logh}
\partial_t \log H_b(t) \leq N (1 + \frac{N_b(t)}{t+b}),\ \ \ \ \ 
\frac{N_b(t)}{t+b}  \leq N(1 + \partial_t \log H_b(t)).
\end{equation}
We next note that if $0<t_1 \leq t_2$ and $t_2 + b \leq \beta$, then \eqref{monb} implies (with the same universal constant $C>0$)
\begin{align}\label{ino1}
N_b(t_1) \leq e^{C}(N_b(t_2)+1).
\end{align}
Now, it is easy to verify that if $t+b \leq  \frac{\beta}{12}$ then $\frac{1}{\log{(3/2)}} \int_{\beta/4}^{\beta/2} \frac{dw}{w+b} \ge 1$, and that $\log\frac{\beta/2 + b}{\beta/4 + b}\le \log(7/3)$. From \eqref{ino1} and the second inequality in \eqref{logh} we thus have for $t+b \leq  \frac{\beta}{12}$ 
\begin{align*}
N_b(t) &\leq e^{C}(N_b(\beta/4)+1) \leq e^C\left\{\frac{1}{\log(3/2)}\int_{\beta/4}^{\beta/2} N_b(\beta/4)\frac{dw}{w+b} +1\right\}
\notag\\
& \leq e^C\left\{\frac{1}{\log(3/2)}\int_{\beta/4}^{\beta/2} \frac{e^C(N_b(w) +1)
}{w+b}dw + 1\right\}
\notag\\
& \le e^C\left\{\frac{e^C N}{\log(3/2)} \int_{\beta/4}^{\beta/2} (1+\p_w \log H_b(w))dw + \frac{e^C}{\log(3/2)} \log\frac{\beta/2 + b}{\beta/4 + b} + 1\right\}.
\end{align*}
This gives for $t+b \leq  \frac{\beta}{12}$
\begin{equation}\label{lk1}
N_b(t) \le C^\star\left(1 + \log \frac{H_b(\beta/2)}{H_b(\beta/4)}\right),
\end{equation}
for an appropriate constant $C^\star>0$. 
Now, by \eqref{kju} in Lemma \ref{L:below} we obtain
\[
H_b(\beta/4) \geq  N^{-1} e^{-\frac{1}{(\beta/4)+b}} \int_{\mathbb B_1} U(X,0)^2 y^a dX \ge N^{-1} e^{-\frac{1}{(\beta/4)}} \int_{\mathbb B_1} U(X,0)^2 y^a dX.
\]
Since $e^{-\frac{1}{(\beta/4)}} = (N\theta)^{-4N}$, we have found
\begin{equation}\label{lk4}
H_b(\beta/4) \ge N^{-1} (N\theta)^{-4N} \int_{\mathbb B_1} U(X,0)^2 y^a dX.
\end{equation}
On the other hand, we have from  \eqref{Hb} and the support property of $\phi$
\begin{align*}
& H_b(\beta/2)  \le  \int_{\mathbb B_{7/2}} U(X,\beta/2)^2 \G(X,\beta/2+b) y^a dX
\\
& \le ||U(\cdot,\beta/2)||_{L^\infty(\mathbb B_{7/2})}^2 \int_{\R^{n+1}_+} \G(X,\beta/2+b) y^a dX.
\end{align*}
By \eqref{Ga1} and the $L^{\infty}$ estimate in Lemma \ref{reg1} we infer for some $\overline C = \overline C(n,a)>0$
\begin{equation}\label{lk2} 
H_b(\beta/2) \leq \overline C  \int_{\mathbb Q_4} U^2 y^a dXdt.
\end{equation}
Using \eqref{lk4} and \eqref{lk2} in \eqref{lk1}, after possibly adjusting the choice of $N$, we finally obtain \eqref{Nbound} in \emph{\underline{Claim II}}.

With \eqref{Nbound} in hands, if we set $t=0$ in it (keeping \eqref{Nb} in mind) we deduce the following inequality
\begin{align}\label{nk01}
&b \int |\n f(X,0)|^2\G (X,b) y^a dX + b \its V(x,0) f((x,0),0)^2\G ((x,0),b)dx\\ & \leq N\operatorname{log}(N \theta) \int  f(X,0)^2\G (X,b) y^a dX.\notag
\end{align}
If we now use the trace inequality in the first line of \eqref{tracef} with $t=0$ we obtain 
\begin{align*}
& \bigg|\its V(x,0) f((x,0),0)^2\G ((x,0),b)dx\bigg| \leq  ||V||_\infty \its f((x,0),0)^2\G ((x,0),b)dx
\\
& \le C_0 ||V||_\infty b^{-\frac{1+a}{2}}\int f(X,0)^2 \G(X,b) y^a dX + C_0 ||V||_\infty b^{\frac{1-a}{2}} \its |\n f(X,0)|^2 \G(X,b) y^a dX. 
\end{align*}
Using this estimate in \eqref{nk01} we infer
\begin{align*}
& (b - C_0 ||V||_\infty b^{\frac{3-a}2}) \int |\n f(X,0)|^2\G (X,b) y^a dX - C_0 ||V||_\infty b^{\frac{1-a}2} \int f(X,0)^2 \G(X,b) y^a dX
\\
& \le N\operatorname{log}(N \theta) \int  f(X,0)^2\G (X,b) y^a dX.
\end{align*}
 Since $0 <\frac{1-a}{2} <1$, we now choose $b$ small enough to ensure that 
 \begin{equation}\label{nkoo}
 b^{\frac{1-a}{2}} < \frac{1}{2C_0 ||V||_\infty}.
 \end{equation}
Observing that such choice guarantees the following inequality $b-C_0 ||V||_\infty b^{\frac{3-a}{2}} >b/2$,
we obtain
\begin{align*}
& \frac{b}{2} \int |\n f(X,0)|^2\G (X,b) y^a dX - \frac{1}{2} \int f(X,0)^2 \G(X,b) y^a dX
\\
& \leq N\operatorname{log}(N \theta) \int  f(X,0)^2\G (X,b) y^a dX.
\end{align*}
Since by \eqref{Gpole} we have $\G(X,b) = \frac{(4\pi)^{-\frac n2} 2^{-a}}{\Gamma((1+a)/2)} b^{-\frac{n+a+1}{2}} e^{-\frac{|X|^2}{4b}}$, 
the latter estimate can be rewritten as follows
\begin{align*}
& \frac{b}{2} \int |\n f(X,0)|^2 e^{-\frac{|X|^2}{4b}} y^a dX - \frac{1}{2} \int f(X,0)^2 e^{-\frac{|X|^2}{4b}} y^a dX
\\
& \leq N\operatorname{log}(N \theta) \int  f(X,0)^2 e^{-\frac{|X|^2}{4b}} y^a dX.
\notag
\end{align*}
If we now add $\frac{n+5+a}{8} \int f(X,0)^2 e^{-|X|^2/{4b}} y^a dX$ to both sides of the latter inequality, and then multiply the resulting one by $4$, after possibly further adjusting the choice of $N$ we obtain 
\begin{align}\label{dou}
& 2b \int |\n f(X,0)|^2e^{-|X|^2/{4b}} y^a dX + \frac{n+1+a}{2}\int f(X,0)^2 e^{-|X|^2/{4b}} y^a dX  
\\
& \leq N\operatorname{log}(N \theta) \int  f(X,0)^2 e^{-|X|^2/{4b}} y^a dX.
\notag
\end{align} 
Having established \eqref{dou}, we can now appeal to Lemma \ref{escves4} to deduce that for all $r\leq 1/2$ 
\begin{align*}
\int_{\mathbb B_{2r}}f(X,0)^2dX \leq (N \theta)^N\int_{\mathbb B_{r}} f(X,0)^2dX.
\end{align*}
Since $f=\phi U$ and $\phi \equiv 1$ in $\mathbb B_3$, this finally proves (i).

We next prove (ii). With this in mind, if $N>2$ is fixed as in the proof of (i), we take $0<r \leq 1/\sqrt{N \log(N \theta)}$ and now assume that the parameter $b$ satisfy
\begin{equation}\label{b}
0< b \leq \frac{r^2}{16N\operatorname{log}(N\theta)}.
\end{equation}
If we insert the bound \eqref{Nbound} for $N_b(t)$ in the first inequality in \eqref{logh}, we find for $0< t < t+b \leq \frac{1}{N \log(N\theta)}$
\[
\partial_t \operatorname{log}H_b(t) \leq N \operatorname{log}(N \theta)/(t+b).
\]
Integrating on $[0,t]$, and exponentiating the resulting inequality, we obtain for $t < t+b \leq \frac{1}{N \log(N\theta)}$ (after keeping again \eqref{Gpole} in mind)
\begin{align}\label{1}
\int f(X,t)^2 e^{-|X|^2/{4(t+b)}} y^a dX \leq \left(1+\frac{t}{b}\right)^{N\operatorname{log}(N \theta)}\int f(X,0)^2 e^{-|X|^2/{4b}} y^a dX.
\end{align}
We now estimate the integral in the right-hand side as follows
\begin{align}\label{in1}
& \int f(X,0)^2 e^{-|X|^2/{4b}} y^a dX =\int_{\mathbb B_r} f(X,0)^2 e^{-|X|^2/{4b}} y^a dX+\int_{\R^{n+1}_+\setminus \mathbb B_r} f(X,0)^2 e^{-|X|^2/{4b}} y^a dX
\\
& \leq \int_{\mathbb B_r} f(X,0)^2 e^{-|X|^2/{4b}} y^a dX+\frac{16b}{r^2}\int_{\R^{n+1}_+\setminus \mathbb B_r} \frac{|X|^2}{16b} f(X,0)^2 e^{-|X|^2/{4b}} y^a dX
\notag\\
& \leq \int_{\mathbb B_r} f(X,0)^2 e^{-|X|^2/{4b}} y^a dX + \frac{1}{N \operatorname{log}(N\theta)}\int_{\R^{n+1}_+\setminus \mathbb B_r} \frac{|X|^2}{16b} f(X,0)^2 e^{-|X|^2/{4b}} y^a dX,
\notag
\end{align}
where we have used \eqref{b}. The Hardy type inequality in Lemma \ref{escves3} and \eqref{dou} now give
\begin{align*}
&\int \frac{|X|^2}{8b}f(X,0)^2 e^{-|X|^2/4b}y^adX\\ &\leq  2b \int |\n f(X,0)|^2 e^{-|X|^2/4b}y^a dX + \frac{n+a+1}{2} \int f(X,0)^2  e^{-|X|^2/4b} y^a dX\notag\\
& \leq N\operatorname{log}(N\theta)\int f(X,0)^2  e^{-|X|^2/4b} y^a dX\notag\\
& \leq N\operatorname{log}(N\theta)\left(\int_{\mathbb B_r} f(X,0)^2  e^{-|X|^2/4b} y^a dX +\int_{\R^{n+1}\setminus \mathbb B_r} f(X,0)^2 e^{-|X|^2/{4b}} y^a dX\right)\notag\\
& \leq N\operatorname{log}(N\theta)\int_{\mathbb B_r} f(X,0)^2  e^{-|X|^2/4b} y^a dX + N\operatorname{log}(N\theta)\frac{8b}{r^2}\int_{\R^{n+1}\setminus \mathbb B_r}\frac{|X|^2}{8b} f(X,0)^2 e^{-|X|^2/{4b}} y^a dX.\notag
\end{align*}
By \eqref{b} again, we see that the last integral in the right-hand side can be estimated as follows
\begin{align*}
& N\operatorname{log}(N\theta)\frac{8b}{r^2}\int_{\R^{n+1}\setminus \mathbb B_r}\frac{|X|^2}{8b} f(X,0)^2 e^{-|X|^2/{4b}} y^a dX \leq  \frac 12 \int \frac{|X|^2}{8b}f(X,0)^2 e^{-|X|^2/4b}y^adX.
\end{align*}
Therefore, it can be absorbed in the left-hand side obtaining the following estimate 
\begin{align*}
\int \frac{|X|^2}{16b}f(X,0)^2 e^{-|X|^2/4b}y^adX  \leq N\operatorname{log}(N\theta)\int_{\mathbb B_r} f(X,0)^2  e^{-|X|^2/4b} y^a dX.
\end{align*}
If we use this inequality to estimate the integral on $\R^{n+1}_+\setminus \mathbb B_r$ in the right-hand side of \eqref{in1} we reach the following crucial conclusion
\begin{align*}
\int f(X,0)^2 e^{-|X|^2/{4b}} y^a dX \leq 2\int_{\mathbb B_r} f(X,0)^2  e^{-|X|^2/4b} y^a dX.
\end{align*}
Using this estimate in \eqref{1} we obtain for $0<t \le t+b \leq \frac{1}{N \log(N\theta)}$
\begin{align}\label{in4}
\int f(X,t)^2 e^{-|X|^2/{4(t+b)}} y^a dX \leq 2\left(1+\frac{t}{b}\right)^{N\operatorname{log}(N \theta)}\int_{\mathbb B_r} f(X,0)^2  e^{-|X|^2/4b} y^a dX.
\end{align}
We now choose $b = \frac{r^2}{16N\log(N\theta)}$, so that \eqref{b} is still satisfied, and assuming that $4 r^2 \leq \frac{1}{16N\log(N\theta)}$ we let $0<t \leq 4r^2$. Keeping in mind that   $f=U$ in $\mathbb B_{3}$, and that by our choice of $b$ on the set $\mathbb B_{2r}$ the following bound from below holds $e^{-|X|^2/{4(t+b)}} \ge (N\theta)^{-16N}$, we deduce  from \eqref{in4} the following inequality
\begin{align*}
\int_{\mathbb B_{2r}} U^2(X,t)y^adX \leq 2 (N\theta)^{16N} \left(1+\frac{t}{b}\right)^{N\log(N \theta)} \int_{\mathbb B_r}U^2(X,0)y^adX.
\end{align*}
Integrating this inequality with respect to $t\in [0,4r^2]$ and using our choice $b= \frac{r^2}{16N\operatorname{log}(N\theta)}$ finally gives (ii). The proof of (iii) follows from (ii) using (i) and  the $L^{\infty}$ bound of $U$ in $\mathbb B_{r/2}\times \{0\}$ in terms of $\int_{\mathbb Q_r} U^2 y^a dXdt$ in \cite[Theorem 5.1]{BG}. 
\end{proof}
With Theorem \ref{db1} in hand we now proceed with the proof of our main result.
\begin{proof}[Proof of Theorem \ref{main}]
Our first crucial step is to show that if $u \in \operatorname{Dom}(H^s)$ solves \eqref{e0} in $B_1 \times (-1, 0]$ and vanishes  to infinite order at $(0,0)$, then for the solution $U$ to the extension problem \eqref{exprob} we must have 
\begin{equation}\label{crucial}
U(X,0) \equiv 0, \ \ \ \ \ \ \ \text{for every}\ X\in \mathbb B_1.
\end{equation}
We argue by contradiction and assume that \eqref{crucial} is not true. Consequently, \eqref{ass} does hold and therefore we can use the results in Section \ref{s:main}. In particular, from \eqref{ass} and (i) in Theorem \ref{db1} it follows that $\int_{\mathbb B_r} U(X, 0)^2 y^a dX > 0$ for all $0<r \leq \frac{1}{2}$. From this fact and the continuity of $U$ up to the thin set $\{y=0\}$ we deduce that
\begin{equation}\label{nzero}
\int_{\mathbb Q_r} U^2 y^a dX dt>0, \end{equation} for all $0< r \leq 1/2$. Moreover, the  inequality (iii) in Theorem \ref{db1} holds, i.e. there exist $r_0$ and $C$ depending on $\theta$ in \eqref{theta} such that for all $r\leq r_0$ one has
\begin{equation*}
\int_{\mathbb Q_r} U^2 y^a dX dt \leq C \int_{\mathbb Q_{r/2}} U^2 y^a dXdt.
\end{equation*}
From this doubling estimate we can derive in a standard fashion the following inequality for all $r \leq \frac{r_0}{2}$
\begin{equation*}
\int_{\mathbb Q_r} U^2 y^a dX dt  \geq \frac{r^{L}}{C}   \int_{\mathbb Q_{r_0}} U^2 y^a dX dt,
\end{equation*}
where $L= \operatorname{log}_2 C$. Letting $c_0=\frac{1}{C}\int_{\mathbb Q_{r_0}} U^2 y^a dX dt$, and noting that $c_0>0$ in view of \eqref{nzero}, we can rewrite the latter inequality as
\begin{equation}\label{fvan1}
\int_{\mathbb Q_r} U^2 y^a dX dt  \geq c_0 r^{L}.
\end{equation}
Let now $r_j\searrow 0$ be a sequence such that $r_j \leq r_0$ for every $j\in \mathbb N$, and define
\[
U_j(X,t) = \frac{U(r_jX, r_j^2 t)}{\bigg(\frac{1}{r_j^{n+3+a}}\int_{\mathbb Q_{r_j}} U^2 y^a dX dt\bigg)^{1/2}}. 
\]
Note that on account of  \eqref{nzero} the functions $U_j$'s are well defined. Furthermore, a change of variable gives for every $j\in \mathbb N$
\begin{equation}\label{bound}
\int_{\mathbb Q_1} U_j^2 y^a dX dt=1.\end{equation}
This fact and the above doubling property imply for all $j$
\begin{equation}\label{nondeg}
\int_{\mathbb Q_{1/2}} U_j^2 y^a dX dt \geq C^{-1}.
\end{equation}
Moreover $U_j$ solves the following problem in $\mathbb Q_1$
\begin{equation}\label{expb1}
\begin{cases}
\D( y^a \nabla U_j) + y^a \partial_t U_j=0,
\\
\py U_j ((x,0),t) = r_j^{1-a} V(r_jx, r_j^2 t) U_j ((x,0),t).
\end{cases}
\end{equation}
From \eqref{bound} and the regularity estimates in Lemma \ref{reg1} we infer that, possibly passing to a subsequence which we continue to indicate with $U_j$, we have $U_j \to U_0$ in $H^{\alpha}(\mathbb Q_{3/4})$ up to $\{y=0\}$ and also $\py U_j ((x,0),t) \to \py U_0 ((x,0),t)$ uniformly in $\mathbb Q_{3/4} \cap \{y=0\}$. Consequently, from \eqref{expb1} we infer that the blowup limit $U_0$ solves in $\mathbb Q_{3/4}$ 
\begin{equation}\label{expb2}
\begin{cases}
\D( y^a \nabla U_0) + y^a \partial_t U_0=0,
\\
\py U_0 ((x,0), t) = 0.
\end{cases}
\end{equation}
We now observe that a change of variable and \eqref{fvan1} give
\[
\int_{Q_1}  U_j((x,0),t)^2 dxdt \le c_0^{-1} r_j^{a+1-L} \int_{Q_{r_j}} U((x,0),t)^2 dx dt.
\]
Since by \eqref{exprob} we have $U((x,0),t) = u(x,t)$, by the assumption that $u$ vanishes to infinite order in the sense of \eqref{van1} we infer that 
\[
\int_{Q_1}  U_j((x,0),t)^2 dxdt  \to 0.
\]
Again, since $U_j \to U_0$ uniformly in $\mathbb Q_{1/2}$ up to $\{y=0\}$, we deduce that it must be $U_0 \equiv 0$ in $\mathbb Q_{1/2} \cap \{y=0\}$. Moreover, since $U_0$ solves the problem \eqref{expb2}, we can now apply the weak unique continuation result in Proposition \ref{wucp}  to infer that  $U_0 \equiv 0$ in $\mathbb Q_{1/2}$. On the other hand, from the uniform convergence  of $U_j$'s in $\mathbb Q_{1/2}$ and  the non-degeneracy estimate \eqref{nondeg} we also have
\begin{equation}\label{nondeg1}
\int_{\mathbb Q_{1/2}} U_0^2 y^a dX dt \geq C^{-1},
\end{equation}
and thus $U_0\not\equiv 0$ in $\mathbb Q_{1/2}$. This contradiction leads to the conclusion that \eqref{crucial} must be true.  We now note that, away from the thin set $\{y=0\}$, $U$ solves a uniformly parabolic PDE with smooth coefficients and vanishes identically in the half-ball $\mathbb B_1$.  We can thus appeal to \cite[Theorem 1]{AV} to assert that $U$ vanishes to infinite order both in space and time in the sense of \eqref{van1} at every $(X,0)$ for $X \in \mathbb B_1$. At this point, we can  use the strong unique continuation  result  in   \cite[Theorem 1]{EF}  to finally conclude that $U(X,0) \equiv 0$ for $X\in \R^{n+1}_+$. Letting $y=0$, this implies $u(x,0) = U((x,0),0) \equiv 0$ for $x\in \R^n$. This completes the proof of the theorem. 
\end{proof}


\begin{thebibliography}{99}





	


\bibitem{AV}
G. Alessandrini \& S. Vessella, \emph{Remark on the strong unique continuation property for parabolic operators.} Proc. Amer. Math. Soc. \textbf{2}~(2004) 499-501.
	

	
\bibitem{BDGP}
A. Banerjee, D. Danielli, N. Garofalo \& A. Petrosyan, \emph{The structure of the singular set in the thin obstacle problem for degenerate parabolic equations},
Calc. Var. Partial Differential Equations, \textbf{60}~(2021), no. 3, Paper No. 91, 52 pp.



\bibitem{BG}
A. Banerjee \& N. Garofalo, \emph{Monotonicity of generalized frequencies and the strong unique continuation property for fractional parabolic equations}, Adv. Math. 336 (2018), 149-241.






\bibitem{CS}
L. Caffarelli \& L. Silvestre, \emph{An extension problem related to the fractional Laplacian}, Comm. Partial Differential Equations 32~(2007), no. 7-9, 1245-1260.
	
	
	

\bibitem{EF}
L. Escauriaza \& F. Fernandez, \emph{Unique continuation for parabolic operators}. Ark. Mat. \textbf{41}~ (2003), no. 1, 35-60.

\bibitem{EFV}
L. Escauriaza, F. Fernandez \& S. Vessella, \emph{Doubling properties of caloric functions}, Appl. Anal. \textbf{85}~ (2006), no. 1-3, 205-223.


\bibitem{FF}
M. Fall \& V. Felli, \emph{Unique continuation property and local asymptotics of solutions to fractional elliptic equations}, Comm. Partial Differential Equations, \textbf{39}~(2014), 354-397.

\bibitem{Gft}
N. Garofalo, \emph{Fractional thoughts}. New developments in the analysis of nonlocal operators, 1-135, Contemp. Math., 723, Amer. Math. Soc., Providence, RI, 2019.


\bibitem{Gcm}
N. Garofalo, \emph{Two classical properties of the Bessel quotient $I_{\nu+1}/I_\nu$ and their implications in pde's}. Advances in harmonic analysis and partial differential equations, 57-97, Contemp. Math., 748, Amer. Math. Soc.,  Providence, RI, 2020. 

\bibitem{GL}
N. Garofalo \& F. Lin, \emph{Monotonicity properties of variational integrals, $A_p$ weights and unique continuation}, Indiana Univ. Math. J. \textbf{35}~(1986),  245-268.

\bibitem{GL2}
N. Garofalo \& F. Lin, \emph{Unique continuation for elliptic operators: a geometric-variational approach}. Comm. Pure Appl. Math. \textbf{40}~(1987), no. 3, 347-366.







\bibitem{Jr1}
B. F. Jones, \emph{Lipschitz spaces and the heat equation}. J. Math. Mech. \textbf{18}~(1968/69), 379-409. 

\bibitem{Jr}
B. F. Jones, \emph{A fundamental solution for the heat equation which is supported in a strip}, J. Math. Anal. Appl. \textbf{60}~(1977), 314-324. 

\bibitem{LLR}
R. Lai, Y. Lin \& A. Ruland, \emph{The Calderón problem for a space-time fractional parabolic equation},  SIAM J. Math. Anal. 52 (2020), no. 3, 2655-2688. 
	
\bibitem{Le}
N. N. Lebedev, \emph{Special functions and their applications}. Revised edition, translated from the Russian and edited by R. A. Silverman. Unabridged and corrected republication. Dover Publications, Inc., New York, 1972.

	
\bibitem{Li}
G. Lieberman, \emph{Second order parabolic differential equations},  World Scientific Publishing Co., Inc., River Edge, NJ, 1996. xii+439 pp. ISBN: 981-02-2883-X.	
	
	



\bibitem{NS}
K. Nystr\"om \& O. Sande, \emph{Extension properties and boundary estimates for a fractional heat operator}, Nonlinear Analysis, 140~(2016), 29-37.




  
  \bibitem{Po}
C. C. Poon, \emph{Unique continuation for parabolic equations}, Comm. Partial Differential Equations \textbf{21}~ (1996), no. 3-4, 521-539.  



\bibitem{Ru}
A. R\"uland, \emph{On some rigidity properties in PDEs}, Dissertation, Rheinischen Friedrich-Wilhelms-Universit\"at Bonn, 2013.

\bibitem{Ru1}
A. R\"uland, \emph{Unique continuation for fractional Schr\"odinger equations with rough potentials}, Comm. Partial Differential Equations \textbf{40}~ (2015), no. 1, 77-114.

\bibitem{Samko}
S. G. Samko, \emph{Hypersingular integrals and their applications}. Analytical Methods and Special Functions, 5. Taylor \& Francis Group, London, 2002. xviii+359 pp. 

\bibitem{Sam}
C. H. Sampson, \emph{A characterization of parabolic Lebesgue spaces}. 
Thesis (Ph.D.)-Rice University. 1968. 91 pp. 


\bibitem{ST} P.~R. Stinga \&  J.~L. Torrea, 
\emph{Regularity theory and extension problem for fractional nonlocal parabolic equations and the master equation}, SIAM J. Math. Anal. 49 (2017), 3893--3924.



\end{thebibliography}
\end{document}